\documentclass{amsart}

\usepackage{amssymb}

\usepackage[foot]{amsaddr}

\usepackage[cmtip,all]{xy}
\SilentMatrices % less verbose output

\copyrightinfo{2013}{J. D. Christensen and G. Wang}

\newtheorem{thm}{Theorem}[section]
\newtheorem{lem}[thm]{Lemma}
\newtheorem{pro}[thm]{Proposition}
\newtheorem{cor}[thm]{Corollary}
\newtheorem{conj}[thm]{Conjecture}
\newtheorem{question}[thm]{Question}

\theoremstyle{definition}
\newtheorem{de}[thm]{Definition}
\newtheorem{ex}[thm]{Example}

\theoremstyle{remark}
\newtheorem{rmk}[thm]{Remark}

\numberwithin{equation}{section}

\makeatletter
\def\subsection{\@startsection{subsection}{2}%
  \z@{.5\linespacing\@plus.7\linespacing}{.5\linespacing}%
  {\normalfont\bfseries}}
\makeatother

\let\leq\leqslant
\let\geq\geqslant

\usepackage[letterpaper,margin=1.25in]{geometry}
\newcommand{\dosomething}[1]{\textbf{[#1]}}
\newcommand{\donothing}[1]{}
\newcommand{\Comment}{\donothing}
\newcommand{\Commentson}{\renewcommand{\Comment}{\dosomething}}

\Commentson

\usepackage{verbatim}

\usepackage{enumerate}

\usepackage{tikz}
\tikzstyle{pre}=[<-,shorten <=2pt,shorten >=1.3pt,>=stealth,semithick]
\tikzstyle{post}=[->,shorten >=2pt,shorten <=1.3pt,>=stealth,semithick]
\tikzstyle{dot}=[circle, draw, fill=black!50, inner sep=0pt, minimum width=3pt]
\tikzstyle{X}=[above left=-2.5pt,font=\footnotesize]
\tikzstyle{Y}=[above right=-2.5pt,xshift=-1.8pt,font=\footnotesize]
\newcommand{\Dstart}{\path (0,0) node[dot] {}}
\newcommand{\Dend}{;}
\newcommand{\Draw}[4]{+#3 node[dot](next) {} +(0,0) edge[#2] node[#4] {#1} (next.center) ++#3}
\newcommand{\Dnext}[3]{\Draw{#1}{#2}{#3}{#1}}
\newcommand{\DXo}[1]{\Dnext{X}{post}{#1}}
\newcommand{\DYo}[1]{\Dnext{Y}{post}{#1}}
\newcommand{\Dxo}[1]{\Dnext{X}{pre}{#1}}
\newcommand{\Dyo}[1]{\Dnext{Y}{pre}{#1}}
\newcommand{\DX} {\DXo{(-.7,-1)}}
\newcommand{\DXw}{\DXo{( -1,-1)}}
\newcommand{\Dx} {\Dxo{( .7, 1)}}
\newcommand{\Dxw}{\Dxo{(  1, 1)}}
\newcommand{\DY} {\DYo{( .7,-1)}}
\newcommand{\DYw}{\DYo{(  1,-1)}}
\newcommand{\Dy} {\Dyo{(-.7, 1)}}
\newcommand{\Dyw}{\Dyo{( -1, 1)}}
\newenvironment{Dtikz}{\begin{center}\begin{tikzpicture}[scale=0.6]}
  {\end{tikzpicture}\end{center}}

\usetikzlibrary{matrix}

\newcommand{\st}{\mid}

\newcommand{\stmod}[1]{\mathsf{stmod}(#1)}
\newcommand{\StMod}[1]{\mathsf{StMod}(#1)}
\newcommand{\modu}[1]{\mathsf{mod}(#1)}
\newcommand{\Mod}[1]{\mathsf{Mod}(#1)}

\newcommand{\sHom}{\underline{\mathrm{Hom}}}
\newcommand{\Hom}{{\mathrm{Hom}}}
\newcommand{\PHom}{{\mathrm{PHom}}}
\newcommand{\sEnd}{\underline{\mathrm{End}}}

\newcommand{\thick}[1]{\mathsf{Thick}\langle #1 \rangle}

\newcommand{\loc}[1]{\mathsf{Loc}\langle #1 \rangle}

\newcommand{\len}{\mathrm{len}}
\newcommand{\gel}{\mathrm{gel}}
\newcommand{\gl}{\mathrm{gl}}
\newcommand{\rl}{\text{rad\,len }}
\newcommand{\gn}{\text{ghost\,num }}
\newcommand{\gen}{\text{gen\,num }}

\newcommand{\up}{{\uparrow}}
\newcommand{\down}{{\downarrow}}

\newcommand{\iso}{\cong}

\newcommand{\im}{\mathrm{im}}
\newcommand{\soc}{\mathrm{soc}}
\newcommand{\rad}{\mathrm{rad}}

\newcommand{\T}{\mathbb{T}}

\newcommand{\cF}{\mathcal{F}}
\newcommand{\cG}{\mathcal{G}}
\newcommand{\cI}{\mathcal{I}}

\newcommand{\cP}{\mathcal{P}}

\newcommand{\cW}{\mathcal{W}}

\newcommand{\dfn}[1]{\textbf{#1}}

\newcommand{\xqedhere}[1]{\rlap{\hbox to#1{\hfil\llap{\ensuremath{\qed}}}}}

\usepackage{ifpdf}
\ifpdf  \usepackage[pdftex,bookmarks=false]{hyperref}
\else   \usepackage[hypertex]{hyperref}
\fi

\begin{document}

\title{Ghost numbers of group algebras}

\author{J. Daniel Christensen}
\address{Department of Mathematics\\
University of Western Ontario\\
London, ON N6A 5B7, Canada}
\email{jdc@uwo.ca}

\author{Gaohong Wang}
\email{gwang72@uwo.ca}

% 20: Group theory and generalizations
%   20C: Representation theory of groups
%     20C20: Modular representations and characters
%   20J: Connections with homological algebra and category theory
%     20J06: Cohomology of groups
% 16: Associative rings and algebras
%   16G: Representation theory of rings and algebras
%     16G70: Auslander-Reiten sequences (almost split sequences) and Auslander-Reiten quivers
% 18: Category theory; homological algebra
%   18E: Abelian categories
%     18E30: Derived categories, triangulated categories
% 55: Algebraic topology
%   55P: Homotopy theory
%     55P42: Stable homotopy theory, spectra
%     55P99: None of the above, but in this section

\date{June 27, 2014}


\begin{abstract}
Motivated by Freyd's famous unsolved problem in stable homotopy theory,
the generating hypothesis for the stable module category of a finite group is the
statement that if a map in the thick subcategory generated by the trivial representation
induces the zero map in Tate cohomology, then it is stably trivial.
It is known that the generating hypothesis fails for most groups.
Generalizing work done for $p$-groups, 
we define the ghost number of a group algebra, which is a natural number that measures
the degree to which the generating hypothesis fails.
We describe a close relationship between ghost numbers and Auslander-Reiten triangles,
with many results stated for a general projective class in a general triangulated category.
We then compute ghost numbers and bounds on ghost numbers for many families of $p$-groups,
including abelian $p$-groups, the quaternion group and dihedral $2$-groups,
and also give a general lower bound in terms of the radical length, the first 
general lower bound that we are aware of.
We conclude with a classification of group algebras of $p$-groups with small ghost number
and examples of gaps in the possible ghost numbers of such group algebras.
\end{abstract}

\maketitle

\tableofcontents

\section{Introduction}

In modular representation theory, the Tate cohomology functor plays a central role, 
analogous to the role that the homotopy groups play in homotopy theory.  
Thus it is natural to study the \emph{kernel} of Tate cohomology, that is,
the collection of maps which induce the zero map in Tate cohomology. 
These maps are called \dfn{ghosts}, and are the topic of the present paper.

Let $G$ be a finite group, and let $k$ be a field whose characteristic $p$ divides the order of $G$.
We write $\StMod{kG}$ for the stable module category of $kG$, the triangulated category formed
from the module category
by killing the projectives, $\stmod{kG}$ for the full subcategory of finitely generated modules,
and $\thick k$ for the thick subcategory generated by the trivial representation,
a full subcategory of $\stmod{kG}$.
(See Section~\ref{se:gpc} for complete definitions and further background.)

The generating hypothesis (GH) for the stable module category is the statement
that if a map in $\thick k$ induces the zero map in Tate cohomology, then
it is stably trivial.  Using the terminology of the first paragraph, this
is equivalent to saying that all ghosts in $\thick k$ are trivial.
This problem is motivated by Freyd's famous conjecture in homotopy
theory~\cite{freydGH}, which is still open.

By work of Benson, Carlson, Chebolu, Christensen and Min\'a\v{c} 
(Theorem~\ref{th:bcccm} below), 
it is known that the generating hypothesis fails for most groups.
The extent to which it fails is measured by the \dfn{ghost number} of $kG$,
which is the smallest number $n$ such that every composite of $n$ ghosts
in $\thick k$ is stably trivial.
With this terminology, the generating hypothesis is the statement
that the ghost number is one.
The ghost number was studied for $p$-groups in~\cite{Gh in rep}, but even for $p$-groups
it was found to be difficult to calculate, and in most cases only
crude bounds are known.
It is a long-term goal to understand whether this invariant has a simple
description in terms of other invariants of $kG$.

In the present paper
we develop new techniques for the study of ghost numbers and use them to
make new computations in many cases.
For example, we make the first computations of the ghost numbers of 
group algebras of wild representation type at an odd prime 
($k(C_3 \times C_3)$ and others mentioned in the detailed summary below)
as well as the first computations of the ghost numbers of non-abelian
group algebras (the dihedral 2-groups).

\theoremstyle{plain}
\newtheorem*{th:C_3xC_3}{Theorem~\ref{th:C_3xC_3}}
\begin{th:C_3xC_3}
Let $G = C_3 \times C_3$, and
let $k$ be a field of characteristic $3$.
Then the ghost number of $kG$ is $3$.
\end{th:C_3xC_3}

\newtheorem*{co:D_4q}{Corollary~\ref{co:D_4q}}
\begin{co:D_4q}
Let $k$ be a field of characteristic 2.
Then the ghost number of $kD_{4q}$ is $q+1$ for all $q \geq 1$,
where $D_{4q}$ denotes the dihedral $2$-group of order $4q$, with $q$ a power of $2$.
\end{co:D_4q}

Our followup paper~\cite{CW} builds on the work here in order
to compute the ghost numbers of non-$p$-groups.
For example, using Corollary~\ref{co:D_4q}, we are
able to compute the ghost number of an arbitrary dihedral group
at the prime~2 in~\cite{CW}.

We also give many new bounds on ghost numbers, including lower bounds,
which are generally difficult to come by.  We highlight two such 
results here:

\newtheorem*{co:lower-bound}{Corollary \ref{co:lower-bound}}
\begin{co:lower-bound}
Let $G$ be a $p$-group, and let $k$ be a field of characteristic $p$. Then
\[ \frac{1}{2} \rl kG \leq \gn kG \leq \gen kG < \rl kG, \]
when $p$ is even, and
\[ \frac{1}{3} \rl kG \leq \gn kG \leq \gen kG < \rl kG, \]
when $p$ is odd.
\end{co:lower-bound}

\newtheorem*{pr:lower-bound}{Proposition~\ref{pr:lower-bound}}
\begin{pr:lower-bound}
Let $k$ be a field of characteristic $p$.
If $G$ is a group of order $p^r\!$, then
the ghost number of $kG$ is at least $(r-1)(p-1)+1$.
\end{pr:lower-bound}

Our work also includes results which are quite general, in some cases
applying to any projective class in any triangulated category.

\medskip

We now give a detailed summary of the contents of the paper.
We begin in Section~\ref{ss:stmod} by reviewing the stable module category.
In Section~\ref{ss:GH} we recall the statement of the generating hypothesis 
in this situation and state the result of Benson, Carlson, Chebolu, Christensen and Min\'a\v{c} 
that says that the GH fails unless the Sylow $p$-subgroup of $G$ is $C_2$ or $C_3$.
The ghost number, which measures the degree to which the GH fails,
is best studied using the idea of a projective class, so we introduce
projective classes and their associated invariants in Section~\ref{ss:pc}.
Briefly, a projective class consists of a collection $\cP$ of objects
(thought of as ``projective'' building blocks) and an ideal $\cI$ of morphisms
(the maps invisible to the objects in $\cP$) satisfying some axioms.
In this section, we also define the invariants we will study.
The \dfn{ghost length} of a $kG$-module $M$ is the smallest number $n$
such that every composite of $n$ ghosts in $\thick k$ starting from $M$
is stably trivial.  The ghost number that we introduced previously is
the supremum of the ghost lengths of modules in $\thick k$.
In addition, it is convenient to define the \dfn{generating length} of $M$
to be the smallest number $n$ such that every composite of $n$ ghosts
in $\StMod{kG}$ starting from $M$ is trivial, 
and the \dfn{generating number} of $kG$ to be the supremum of the
generating lengths of modules in $\thick k$.

In Section~\ref{se:AR} we present a variety of new results, many of
which hold for arbitrary (stable) projective classes in arbitrary triangulated categories.
For example, in Section~\ref{ss:length and tri}, we give new bounds on the length of an
object in a triangle in terms of the lengths of the other two objects and
the filtration of the connecting homomorphism in the powers of the ideal.
Then, in Section~\ref{ss:length and AR}, we show that the connecting map
$\gamma : Z \to \Sigma X$ in an Auslander-Reiten triangle,
which we call the \dfn{almost zero map},
has a remarkable property:  if $(\cP, \cI)$ is any projective class such
that there is a nonzero map from $Z$ in $\cI^k$, then $\gamma$ is in $\cI^k$.
So the almost zero map is in some sense a universal example of a non-zero map from $Z$.
We specialize to the case of the stable module category in Section~\ref{ss:irr},
where we show that the \dfn{heart} of an indecomposable module $M$
(the fibre of the almost zero map) has length which differs by at most one
from $M$, with respect to any projective class.
We also show that this is true for any summand of the heart,
by showing that the lengths of the domain and codomain of any
irreducible map differ by at most one.
We finish Section~\ref{se:AR} with Section~\ref{ss:gl}, which describes 
the extent to which our results hold for the ghost length, the invariant
used in defining the ghost number.

Section~\ref{se:pgroups} contains detailed computational results on
the ghost numbers of $p$-groups.
We begin by recalling some background results in Section~\ref{ss:bg}, 
such as the fact that the ghost number of $kG$ is less than the nilpotency
index of the Jacobson radical, as well as the fact that multiplication
by $x - 1$, where $x$ is a central element of $G$, is always a ghost.
In Section~\ref{ss:gen-and-socle-length} we show that the generating length invariant
is in a precise sense a stabilized version of the socle length, and show that
if these are equal for a module $M$, the same is true for $\rad(M)$ and $M/\soc(M)$.
This follows from a general result involving nested unstable projective classes in
a triangulated category.
We begin our computations in Section~\ref{ss:gn-abelian}, where we study the
ghost numbers of abelian $p$-groups.  
The main result here is an improved lower bound on the ghost number.
This follows from a result giving a lower bound on the ghost length
of induced modules for general $p$-groups.
We also compute the exact ghost length of many modules over abelian $p$-groups.
In Section~\ref{ss:gn-Q8} we show that the ghost number for the quaternion
group $Q_8$ is $3$ or $4$, improving the existing lower bound by $1$.
In Section~\ref{ss:cyclic sub}, we compute the ghost length and generating length
of certain modules induced up from a cyclic normal subgroup of a $p$-group,
generalizing the technique used for $Q_8$.
This is used in the same section to show that the ghost number and the 
radical length are within a factor of three of each other for any $p$-group.
More precisely, we show that 
$(\rl kG)/3 \leq \gn kG < \rl kG$
for $p$ odd, the first general lower bound we are aware of.
For $p = 2$, the factor of $3$ is replaced with a factor of $2$.
We also use the induction result in Section~\ref{ss:gn-D4q},
where we show that the ghost number of the dihedral $2$-group
$D_{4q}$ of order $4q$ is exactly $q+1$.
This is the longest section of the paper.
That the ghost length is at least $q+1$ follows immediately
from the induction result of the previous section, but that
it is no more than $q+1$ requires using the classification of
$kD_{4q}$-modules.
In Section~\ref{ss:gn-C3xC3} we show that the ghost number of
$k(C_3 \times C_3)$ is exactly 3.  
While $k(C_3 \times C_3)$-modules are not classifiable, we make
use of the fact that certain quotients can be classified.
Our argument also shows that the ghost number of the group algebra 
$k(C_{p^r} \times C_{p^s})$, for $p^r, p^s > 2$, is at most $p^r + p^s - 3$.
It follows that the ghost number of $k(C_3 \times C_{3^s})$ is $3^s$
and that the ghost number of $k(C_4 \times C_{2^s})$ is $2^s + 1$.
We end the paper with Section~\ref{ss:possible}, in which we
give complete lists of the group algebras of $p$-groups with
ghost numbers 1, 2 or 3, with the possible exception of $kQ_8$.
We also prove that for each prime $p$ there are gaps in the
possible ghost numbers that can occur, and state a conjecture
related to this.

\medskip

Our work also raises various questions, which we briefly summarize here.  
We have shown that the generating number is a stabilized version of
the radical length, and that both the generating number and ghost number
are within a constant factor of the radical length, but it still remains
to fully understand these new invariants and determine whether they have
an exact description in terms of existing invariants.

\begin{question}
How do the ghost number and generating number relate to other 
invariants of the group algebra?
\end{question}

Moreover, in all examples where they have been computed, 
the ghost number and the generating number agree, and we
conjecture that this is always the case.

\begin{conj}
Let $G$ be a finite group, and let $k$ be a field whose characteristic 
divides the order of $G$.
Then the ghost number of $kG$ is equal to the generating number of $kG$.
\end{conj}

We also believe that the ghost number of a general $p$-group is
bounded in the following way, as described in Section~\ref{ss:possible}.

\begin{conj}
Let $k$ be a field of characteristic $p$. 
If $G$ is a $p$-group of order $p^r$, then
\[ \text{ghost number of } k(C_{\!p}^r) \leq
   \text{ghost number of } kG       \leq
   \text{ghost number of } k(C_{p^r}) . \]
\end{conj}

\section{The generating hypothesis and the ghost projective class}
\label{se:gpc}

In this section, we recall background material which provides context to our results
and which we use in our proofs.

\subsection{The stable module category}\label{ss:stmod}

Here we recall the basics of the stable module category. 
A good reference is~\cite{Carlson}.

Let $G$ be a finite group, and let $k$ be a field whose characteristic $p$ divides the order of $G$. 
The \dfn{stable module category} $\StMod{kG}$ is a quotient category of the category $\Mod{kG}$ 
of left $kG$-modules by the ideal of maps that factor through a projective.  
Thus the objects of $\StMod{kG}$ are left $kG$-modules and the hom-sets are 
$\sHom(M,N) = [M,N] := \Hom(M,N)/\PHom(M,N)$, where $\PHom(M,N)$ denotes the stably trivial maps, 
i.e., those that factor through a projective module. 
Two modules $M$ and $N$ are isomorphic in the stable module category if and only if 
they have the same projective-free summands.
In particular, projective modules are isomorphic to zero in the stable module category.
We write $\stmod{kG}$ for the full subcategory of finitely generated $kG$-modules in $\StMod{kG}$. 
(More precisely, we include all modules which are \emph{stably isomorphic} to finitely generated
$kG$-modules.)
% The above is slightly non-standard, but is needed to ensure that 
% (1) stmod is a thick subcategory of StMod
% (2) thick k is contained in stmod

The stable module category is a triangulated category. 
The desuspension $\Omega M$ of a module $M$ is the kernel of any surjection $P \to M$ with $P$ projective.
This is well-defined in the stable module category by Schanuel's Lemma~\cite[Prop.~4.2]{Carlson},
and we write $\tilde{\Omega} M$ for the projective-free summand of $\Omega M$.

The group algebra $kG$ is injective as a module over itself. 
In particular, this implies that projective modules and injective modules coincide in $\modu{kG}$. 
The suspension $\Sigma N$ of a module $N$ is defined to be the cokernel
of any injection $N \to P$ with $P$ injective.
We will often write $\Omega^{-1} N$ for $\Sigma N$ since $\Omega$ and $\Sigma$
are inverse functors up to natural isomorphism.

Write $k$ for the trivial representation
and $\thick k$ for the thick subcategory generated by $k$,
the smallest full triangulated subcategory of $\StMod{kG}$ 
that is closed under retracts and contains $k$.
This is in fact a full subcategory of $\stmod{kG}$, and plays
a central role in our formulation of the generating hypothesis.
The localizing category generated by $k$, denoted $\loc k$, 
is the smallest full triangulated subcategory of $\StMod{kG}$
that is closed under arbitrary coproducts and retracts and contains $k$. 

\subsection{The generating hypothesis}\label{ss:GH}

An important feature of the stable module category is that 
the Tate cohomology of a $kG$-module $M$ is representable, 
i.e., we have a canonical isomorphism $\hat{H}^n(G, M) \iso [\Omega^n k, M]$. 

We say that the \dfn{generating hypothesis (GH)} holds for the stable module category $\StMod{kG}$ 
if and only if the Tate cohomology functor $\hat{H}^*(G,-)$ restricted to $\thick k$ is faithful.
It has been shown that the GH fails for most group 
algebras~\cite{GH for p,GH split,admit,GH per}.

\begin{thm}[Benson, Carlson, Chebolu, Christensen and Min\'a\v{c}]\label{th:bcccm}
Let $G$ be a finite group, and let $k$ be a field whose characteristic $p$ divides the order of $G$. 
Then the GH holds for $\StMod{kG}$ if and only if 
the Sylow $p$-subgroup $P$ of $G$ is either $C_2$ or $C_3$.
\end{thm}

It is worth pointing out here why we restrict to $\thick k$. 
It is known that whenever the thick subcategory is not all of $\stmod{kG}$,
there are non-projective modules whose Tate cohomology is zero. 
The identity map on such a module is sent to zero by $\hat{H}^*(G,-)$,
so the GH would be trivially false if we included such modules.
Restricting to $\thick k$ prevents this from happening.
In general, the stable module category is generated 
by the simple modules as a triangulated category. 
For a $p$-group $G$, the trivial representation $k$ is the only simple module, 
so we have that $\thick k = \stmod{kG}$ in this case.

We call a map in $\StMod{kG}$ that is in the kernel of the Tate cohomology functor a
$\dfn{ghost}$.  Thus the GH is the statement that all ghosts in $\thick k$ are stably
trivial. 
When the GH fails, the vanishing of composites of ghosts gives a measure
of the failure and leads to invariants of modules and of $kG$.
This is formalized in the idea of a projective class.

\subsection{The ghost projective class}\label{ss:pc}

In this section, we introduce the invariants that are the subject of this paper:
generating length, ghost length, generating number and ghost number.
These are defined using the concept of a projective class:

\begin{de}
Let\/ $\T$ be a triangulated category. 
A \dfn{projective class} in\/ $\T$ consists of a class $\cP$ of objects of\/ $\T$
and a class $\cI$ of morphisms of\/ $\T$ such that:
\begin{enumerate}[(i)]
\item $\cP$ consists of exactly the objects $P$ such that every composite
      $P \to X \to Y$ is zero for each $X \to Y$ in $\cI$,
\item $\cI$ consists of exactly the maps $X \to Y$ such that every composite
      $P \to X \to Y$ is zero for each $P$ in $\cP$.
\item for each $X$ in\/ $\T$, there is a cofibre sequence $P \to X \to Y$ with 
      $P$ in $\cP$ and $X \to Y$ in $\cI$.
\end{enumerate}
In this paper, we make the additional assumption that the projective class is \dfn{stable},
that is, that $\cP$ (or equivalently $\cI$)
is closed under suspension and desuspension.
With slight alterations, most of our results remain true without this assumption, 
but the extra bookkeeping complicates the arguments.
The one exception is that in Section~\ref{ss:gen-and-socle-length} we make use
of an unstable projective class.
%
%The objects in $\cP$ are said to be $\cI$-\dfn{projective} or \dfn{relative projective}, 
%and the maps in $\cI$ are said to be $\cP$-\dfn{null} or \dfn{relative null}.
\end{de}

\begin{rmk} 
It follows from the definition that $\cP$ is closed 
under arbitrary coproducts and retracts, and that $\cI$ is an ideal.
\end{rmk}

We write $\cG$ for the ideal of ghosts in the stable module category, 
and $\cF$ for all retracts of direct sums of suspensions of $k$ in $\StMod{kG}$.
For a module $M \in \StMod{kG}$, since $\hat{H}^n(G,M) \iso [\Omega^n k, M]$,
we can form a map $\oplus \, \Omega^i k \to M$ that is surjective on Tate cohomology 
by assembling sufficiently many homogeneous elements in $\hat{H}^*(G,M)$. 
Completing this map into a triangle in $\StMod{kG}$
\begin{equation}\label{eq:univ} 
\Omega U_M \to \oplus \, \Omega^i k \to M \xrightarrow{\phi_M} U_M, 
\end{equation}
we get a ghost $\phi_M: M \to U_M$. 
The map $\phi_M$ is a (weakly) universal ghost in the sense that every ghost 
out of $M$ factors though it, but the factorization is not necessarily unique. 
It follows easily that $(\cF,\cG)$ forms a projective class in $\StMod{kG}$. 
This is called the \dfn{ghost projective class}.

While the ghost projective class is the focus of this paper, some of
our results apply to any projective class, so we mention two other examples
at this point:
The \dfn{simple ghost projective class} is the projective class whose 
projectives are generated by all simple objects, and it was proposed
for study in~\cite{GH general} as a way to avoid focusing on $\thick k$.
And the \dfn{strong ghost projective class} is the projective class
whose ideal consists of the maps which are ghosts under restriction
to every subgroup.  
The last example has been studied by Carlson, Chebolu and Min\'a\v{c}
in work in progress, and all three of the projective classes defined
above are studied in~\cite{CW}.

For any projective class $(\cP, \cI)$, there is a sequence of 
\dfn{derived projective classes} $(\cP_n,\cI^n)$~\cite{Chr}.
The ideal $\cI^n$ consists of all $n$-fold composites of maps in $\cI$, 
and $X$ is in $\cP_n$ if and only if it is a retract of an object $M$ that
sits inside a cofibre sequence $P \to M \to Q$ with $P \in \cP_1=\cP$ and 
$Q \in \cP_{n-1}$. 
For $n=0$, we let $\cP_0$ consist of all zero objects and $\cI^0$ consist of all
maps in $\T$. 
%And for $n=\infty$, $\cPinf$ is the collection of retracts of
%coproducts of objects in $\cup_n \cP_n$ and $\cI^{\infty}$ is the
%intersection $\cap_n I^n$.
The \dfn{length} $\len_{\cP}(X)$ of an object $X$ of $\T$ with respect to $(\cP, \cI)$
is the smallest $n$ such that $X$ is in $\cP_n$, 
if this exists.
The fact that each pair $(\cP_n, \cI^n)$ is a projective class implies that
the length of $X$ is equal to the smallest $n$ such that every map in $\cI^n$
with domain $X$ is trivial.

The length of a module $M$ with respect to the ghost projective class
is called the \dfn{generating length} of $M$, 
and this exists when $M$ is in $\thick k$.  % And in other cases too.
But since we are interested in the collection $\cG_t$ of ghosts in $\thick k$, 
we also get another invariant.
We describe both invariants, and the associated invariants of $kG$, in the
following definition, generalizing the definition given in~\cite{Gh in rep}
for $p$-groups.

\begin{de} \ 
\begin{itemize}
\item The \dfn{generating length} $\gel(M)$ of $M \in \thick k$ is the smallest 
      $n$ such that $M \in \cF_n$.  That is, $\gel(M) = \len_{\cF}(M)$.
\item The \dfn{ghost length} $\gl(M)$ of $M \in \thick k$ is the smallest
      integer $n$ such that every map in $(\cG_t)^n$
      with domain $M$ is trivial.
%      Note that any map is considered a $0$-fold composite of maps in $\cG_t$.
\item The \dfn{generating number} of $kG$ is the least upper bound of the
      generating lengths of modules in $\thick k$. 
\item The \dfn{ghost number} of $kG$ is the least upper bound of the
      ghost lengths of modules in $\thick k$.
\end{itemize}
\end{de}

With this terminology, the generating hypothesis is the statement
that the ghost number of $kG$ is $1$.

Let $M$ be in $\thick k$.
Since each $(\cF_n, \cG^n)$ is a projective
class and $(\cG_t)^n \subseteq \cG^n$, it follows that 
\begin{align*}
                     \gl(M)  &\leq  \gel(M)
\intertext{and therefore that}
 \text{ghost number of } kG &\leq \text{generating number of } kG.
\end{align*}
When $G$ has periodic Tate cohomology, the coproduct in~\eqref{eq:univ}
can be taken to be finite, and it follows that the ghost projective
class restricts to a projective class in %$\stmod{kG}$ and 
$\thick k$~\cite{Gh in rep}.
This implies that equality holds in this case.
We don't know whether equality holds in general, except for the
trivial observation that 
$M \iso 0$ if and only $\gel(M) = 0$ if and only if $\gl(M) = 0$
and the less trivial fact that $\gel(M) = 1$ if and only if $\gl(M) = 1$
(see Corollary~\ref{cor:M and HM} or~\cite{GH split}).
Thus the GH is equivalent to the generating number of $kG$ being $1$.
See Remark~\ref{re:gl=gel} for further discussion of whether
ghost length equals generating length.

\section{Auslander-Reiten triangles and generating lengths}\label{se:AR}

% Our plan is to use ``Auslander-Reiten triangle'' for this notion.
% If we need to discuss the analog in the module category, we'll call
% it ``Auslander-Reiten sequence''.
In this section, we explain how
Auslander-Reiten triangles (in short, A-R triangles) provide examples of ghosts,
and, more generally, of non-trivial maps in $\cI^n$ for $n$ as large as possible,
for any projective class $(\cP, \cI)$.
This extends the work of~\cite{GH split}, where these triangles are
called ``almost split sequences.''
Because we have in mind applications to other projective classes, in this
section we state many of our results for a general projective class in a general
triangulated category.

In Section~\ref{ss:length and tri}, we give results about the
relationship between the lengths of the objects in a triangle when one of
the maps is in a power $\cI^m$ of the ideal.
In Section~\ref{ss:length and AR}, we recall A-R triangles and prove that the
third map in an A-R triangle is the longest possible non-trivial composite of 
maps in $\cI$ with the given domain. 
In Section~\ref{ss:irr}, we apply these results to the study of
lengths in the stable module category, and also show a close relationship
between lengths and irreducible maps.
Finally, in Section~\ref{ss:gl} we explain the extent to which our
results on generating length are true for ghost length.

\subsection{Relations between the lengths of objects in a triangle}\label{ss:length and tri}

% This section works for unstable projective classes.

Consider a projective class $(\cP,\cI)$ in a triangulated category $\T$. 
Let 
\[ X \xrightarrow{\alpha} Y \xrightarrow{\beta} Z \xrightarrow{\gamma} \Sigma X \]
be a triangle in $\T$, where $X$, $Y$ and $Z$ have finite lengths $k$, $n$ and $l$, 
respectively. 
We know that $n \leq k+l$~\cite{Chr}. 
Rotating the triangle, we also get $l \leq n+k$ and $k\leq n+l$. 
Here we show that when $\gamma$ is in $\cI^m$, one can refine these inequalities
by subtracting $m$ from $l$.
Our methods also show that $n \geq m$.
Note that $\cI^0$ consists of all maps in $\T$.

\begin{lem}\label{le:gel+ar}
Let $(\cP,\cI)$ be a projective class in a triangulated category $\T$, and let
\[ X \xrightarrow{\alpha} Y \xrightarrow{\beta} Z \xrightarrow{\gamma} \Sigma X \]
be a triangle in $\T$, where $X$, $Y$ and $Z$ have finite lengths $k$, $n$ and $l$, 
respectively, and $\gamma \in \cI^m$ with $m \leq l$. 
Then
\[
  \len_{\cP}(Y) = n \leq \max(k-m+l,l).
\]
\end{lem}

Note that if $m \geq l$, then $\gamma$ must be zero, and so the restriction
to $m \leq l$ is natural.  
When $m = l$, the triangle splits, and the lemma says that $n \leq \max(k,l)$. 

\begin{proof}
Let $n' = \max(k,m)$, and let $\phi: Y \to W$ be in $\cI^{n'}$.
Then $\phi \circ \alpha$ is zero (since $n' \geq k$), 
so $\phi$ factors through a map $\tilde{\phi}: Z \to W$.
We claim that $\tilde{\phi}$ is in $\cI^m$. 
Consider the diagram
\[
\xymatrix{
                   &                        & V \ar@{-->}[dl]_{\tilde{\psi}} \ar[d]_{\psi} \\
 X \ar[r]^{\alpha} & Y \ar[r] \ar[d]_{\phi} & Z \ar@{-->}[dl]^{\tilde{\phi}} \ar[r]^{\gamma} & \Sigma X \\
                   & W
}
\]
with $\psi: V \to Z$ being any map from an object $V\in \cP_m$. 
Now $\gamma \in \cI^m$, so $\gamma\circ\psi$ is zero, 
and $\psi$ factors through some map $\tilde{\psi}:V \to Y$. 
Hence $\tilde{\phi} \circ \psi = \phi \circ \tilde{\psi}$ is zero (since $n' \geq m$),
and the claim follows.
If $g: W \to W'$ is in $\cI^{l-m}$,
then $g \circ \tilde{\phi}$ is zero because $Z$ has length $l$.
Then $g \circ \phi$ is zero, meaning that the length of $Y$ is at most $n'+l-m$.
\end{proof}

\begin{lem}\label{le:gl+ar}
Let $(\cP,\cI)$ be a projective class in a triangulated category $\T$, and let
\[ X \xrightarrow{\alpha} Y \xrightarrow{\beta} Z \xrightarrow{\gamma} \Sigma X \]
be a triangle in $\T$, where $X$, $Y$ and $Z$ have finite lengths $k$, $n$ and $l$, 
% In the unstable case, really need that $\Sigma^{-1} Z$ has length l.
respectively, and $\gamma \in \cI^m$ with $m \leq l$. 
% In the unstable case, really need that $\Sigma^{-1} \gamma \in \cI^m$.
% In the unstable case, need to assume that $Z$ has finite length, too, for $n \geq m$.
%
% With the same notation as Lemma~\ref{le:gel+ar}, the following holds:
%
Then
\[
  \len_{\cP}(Y) = n \geq \max(k-l+m,m).
\]
\end{lem}

When $m=l$, this says that $n \geq \max(k,l)$,
so the two lemmas together recover the fact that when the triangle
splits, $n = \max(k,l)$.

\begin{proof}
We prove that the length of $Y$ is at least $k-l+m$. 
The other inequality can be proved similarly. 

Consider a map $\phi: X \to W$ in $\cI^{l-m}$. 
Since $\phi \circ \Sigma^{-1}\gamma$ is in $\cI^l$ and has domain $\Sigma^{-1} Z$
of length $l$, it is zero
and $\phi$ factors through a map $\tilde{\phi} : Y \to W$:
\[
\xymatrix@C+5pt{
 \Sigma^{-1}Z \ar[r]^-{\Sigma^{-1}\gamma} & X \ar[r] \ar[d]_-(0.45){\phi} & Y \ar[r] \ar@{-->}[dl]^-(0.4){\tilde{\phi}} & Z \\
                                          & W & & .
}
\]
Let $g: W \to W'$ be in $\cI^n$. 
Then $g \circ \tilde{\phi}$ is zero because $Y$ has length $n$,
hence any map in $\cI^{n+l-m}$ with domain $X$ is zero. 
This implies that $k \leq n+l-m$, i.e., that $n \geq k-l+m$.
%To prove the second inequality, we take an $n$-fold ghost $\psi: Z \to V$ and 
%further compose it with an $(l-m)$-fold ghost $V \to V'$. A similar argument goes through.
\end{proof}

\subsection{Auslander-Reiten triangles give composites of ghosts}\label{ss:length and AR}

% This section works for unstable projective classes, too.

We begin by recalling the definition.

\begin{de}
  Let $\T$ be a triangulated category.
  A triangle 
  $X \xrightarrow{\alpha} Y \xrightarrow{\beta} Z \xrightarrow{\gamma} \Sigma X$ 
  is called an \dfn{Auslander-Reiten triangle}, if
  \begin{enumerate}[(a)]
  \item $\gamma \neq 0$,
  \item \label{left a-s} any map $X \to Y'$ that is not split monic factors through $\alpha$,
  \item \label{right a-s} any map $Y' \to Z$ that is not split epic factors through $\beta$.
  \end{enumerate}
\end{de}

A map $\alpha$ that is not split monic and satisfies~\eqref{left a-s} is said to be
\dfn{left almost split}.
Dually, a map $\beta$ that is not split epic and satisfies~\eqref{right a-s} is said to be
\dfn{right almost split}.

We know that Auslander-Reiten triangles exist in great generality.

\begin{thm}[Krause,~\cite{Krause}]\label{th:AR}
Let $\T$ be a triangulated category with all small coproducts, 
and suppose that all cohomological functors are representable.
Let $Z$ be a compact object in $\T$ with local endomorphism ring. 
Then there exists an Auslander-Reiten triangle
\[ X \xrightarrow{\alpha} Y \xrightarrow{\beta} Z \xrightarrow{\gamma} \Sigma X. \]
The triangle is unique up to a non-canonical isomorphism. \qed
\end{thm}

\begin{rmk}
Let $\beta$ be the second map in the A-R triangle above.
One can show that, for any endomorphism $g$ of $Y$ with $\beta g=\beta$, the map $g$ is
an isomorphism (see~\cite{Krause}).
We say that the map $\beta$ is \dfn{right minimal} in this case.
Dually, the first map $\alpha$ in an A-R triangle is \dfn{left minimal}.
A map $\beta$ that is right almost split
sits inside an Auslander-Reiten triangle if and only if
it is right minimal~\cite{Krause}.
\end{rmk}

For convenience, we call the map $\gamma$ here the \dfn{almost zero map} 
with domain $Z$.  It is unique up to an automorphism of $\Sigma X$.
The following proposition follows from the definitions and
the earlier lemmas.
% It's also true for infinite length, if interpreted correctly, but I
% don't think we should complicate the statement.

\begin{pro}\label{prop:length r-a-s}
Suppose that $(\cP,\cI)$ is a projective class on a triangulated category $\T$,
and that 
\[ X \xrightarrow{\alpha} Y \xrightarrow{\beta} Z \xrightarrow{\gamma} \Sigma X \]
is a distinguished triangle with $\beta$ right almost split.
If $Z$ has finite length $l$ and $X$ has finite length $k$
with respect to $(\cP,\cI)$, then
the third map $\gamma$ is in $\cI^{l-1}$, and
\begin{alignat*}{2}
   k-1 \leq \len_{\cP}(Y) &\leq k+1, &&\text{ if } k \geq l;\\
   l-1 \leq \len_{\cP}(Y) &\leq l, &&\text{ if } k \leq l-1. 
\end{alignat*}
For any summand $S$ of $Y$, $\len_{\cP}(S) \leq \max(k+1,l)$. 
\end{pro}

\begin{proof}
We test $\gamma$ on all objects $W$ in $\cP_{l-1}$. 
Because $Z$ has larger length than $W$,
a map $\phi:W \to Z$ cannot be split epic, so it factors through $\beta$.
Hence $\gamma \circ \phi$ is zero, which implies that $\gamma \in \cI^{l-1}$. 

The inequalities follow from Lemmas~\ref{le:gel+ar} and Lemma~\ref{le:gl+ar},
with $m = l-1$.  The statement about the summand $S$ follows immediately.
\end{proof}

Note in particular that for any A-R triangle, the almost zero map $\gamma$
is an example of a non-zero map in the largest possible power of the ideal,
for \emph{any} projective class.

% See more discussions in the notes.

In the case when $\T$ is $\StMod{kG}$ with $G$ being a $p$-group,
we know that ghosts and \emph{dual ghosts} coincide~\cite{Gh in rep}.
Hence $\gamma$ non-zero implies that $k \geq l$, and so we are in
the first case of Proposition~\ref{prop:length r-a-s}.

In the next section, we develop these ideas further.

\subsection{Auslander-Reiten triangles, irreducible maps and lengths}\label{ss:irr}

% Note: the results in this section go through in many triangulated categories,
% but to keep things simple we work in stmod here.  See the notes on irreducible maps.
% They might even work for unstable projective classes.

In this section, we focus on the category $\StMod{kG}$, and show that there
is a close relationship between lengths and irreducible maps.

The category $\StMod{kG}$ satisfies the hypotheses on $\T$ in Theorem~\ref{th:AR}, and
its compact objects are precisely the objects of $\stmod{kG}$.
For projective-free $M \in \stmod{kG}$, the stable endomorphism ring $\sEnd(M,M)$ 
being local is equivalent to $M$ being indecomposable.
% Since we are taking \End in the stable category, we can't conclude anything
% about projective summands of M.
In this case, the Auslander-Reiten triangle has the form~\cite[4.12.8]{Benson}
\[ \Omega^2M \xrightarrow{\alpha} H(M) \xrightarrow{\beta} M \xrightarrow{\gamma} \Omega M.\]
The module $H(M)$ is called the \dfn{heart} of $M$, 
and the triangle shows that it is also in $\stmod{kG}$.

The general theory we have set up in the last two sections applies to an \mbox{A-R} triangle
for any projective class $(\cP,\cI)$ on $\StMod{kG}$.
% We don't introduce irreducible maps in arbitrary categories
% because the proof we give below depends heavily on the existence
% of A-R triangles.
As a special case of Proposition~\ref{prop:length r-a-s}, using that
$k=l$ in this case, we get
\begin{cor}\label{cor:M and HM}
 Let $G$ be a finite group, let $k$ be a field whose characteristic divides the order of $G$, 
 and let $(\cP,\cI)$ be a projective class on $\StMod{kG}$.
 Consider the Auslander-Reiten triangle 
 $\Omega^2M \xrightarrow{\alpha} H(M) \xrightarrow{\beta} M \xrightarrow{\gamma} \Omega M$ 
 for some indecomposable non-projective module $M$ in $\stmod{kG}$
 with finite length $l$ with respect to $(\cP,\cI)$.
 Then
 \[
   \len_{\cP}(M)-1 \leq \len_{\cP}(H(M)) \leq \len_{\cP}(M)+1,
 \]
 and $\gamma$ is a non-trivial map in $\cI^{l-1}$.\qed
\end{cor}

As above, we emphasize again that the same map $\gamma: M \to \Omega M$ provides a map
in $\cI^n$ with $n$ maximal for \emph{any} projective class $(\cP,\cI)$.
Put another way, $\gamma$ is in the intersection of all projective class ideals
that contain a non-trivial map from $M$.

\begin{rmk}
One might hope that the heart $H(M)$ always has larger generating length than $M$
when $\gel(M)$ is less than the generating number of $kG$, but
unfortunately this is not true in general.
For example, take $G= C_5 \times C_5$ and $M = k \up _{C_5}^G$.
One can compute that $\gel(M) = \gel(H(M)) = 5$,
while the generating number of $kG$ is at least $6$ (Theorem~\ref{th:ab}).
\end{rmk}

% More generally, let $C_{p^r}$ be a cyclic summand of and abelian group $A$.
% Let $M$ be a non-projective indecomposable $C_{p^r}$-module and
% $M' = M \up ^A$. Then
% $\rl (M') = \rl (H(M'))$.

Let $S$ be an indecomposable non-projective summand of $H(M)$. Then,
clearly, $\len_{\cP}(S) \leq \len_{\cP}(H(M)) \leq \len_{\cP}(M)+1$.
We will show below that $\len_{\cP}(M)-1 \leq \len_{\cP}(S)$ because of
the right minimality of the map $\beta$.

We first need the notion of irreducible map.

\begin{de}
Let $G$ be a finite group, and let $k$ be a field whose characteristic divides the order of $G$.
A map $\lambda: M \to N$ in $\StMod{kG}$ is said to be \dfn{irreducible}
if it is not split monic or split epic, and
for any factorization $\lambda = \nu \circ \mu$,
either $\mu$ is split monic or $\nu$ is split epic.
\end{de}

% We have the following characterization of irreducible maps.
% For finitely generated modules $M=\oplus M_i$ and $N=\oplus N_j$ in $\modu{kG}$,
% with $M_i$ and $N_j$ indecomposable,
% we write $\rad(M,N)$ for the space of maps from $M$ to $N$ such that
% no component $M_i \to N_j$ is an isomorphism.
% Note that the irreducible maps between indecomposable modules $M$ and $N$ are exactly those
% in $\rad(M,N)$ but not in $\rad^2(M,N)$.

Irreducible maps are closely related to Auslander-Reiten triangles:

% It is not hard to see that the map $H(M) \xrightarrow{\beta} M$
% in the A-R triangle is irreducible, and that dually, the map
% $\tilde{\Omega}^2 M \xrightarrow{\alpha} H(M)$ is irreducible. Then one
% can prove the following proposition characterizing irreducible
% maps with indecomposable domains.

\begin{pro}[Auslander and Reiten~\cite{A-R}]\label{prop:irr}
Let $M$ and $N$ be indecomposable non-projective modules in $\stmod{kG}$.
Then a map $f: M \to N$ is irreducible if and only if
the following equivalent conditions are satisfied:
\begin{enumerate}[(a)]
\item $M$ is a summand of $H(N)$ and $f$ is the composite
      $M \to H(N) \xrightarrow{\beta} N$.
\item $N$ is a summand of $\Omega^{-2}H(M)$ and $f$ is the composite      
      $M \xrightarrow{\Omega^{-2} \alpha} \Omega^{-2}H(M) \to N$. \qed
\end{enumerate}
\end{pro}
% In particular, ignoring $f$, we find that for non-projective indecomposables
% M and N, M is a summand of H(N) iff N is a summand of \Omega^{-2} H(M).
% So any summand M of H(N) can build N with two building blocks.
% This shows that the lengths are within a factor of two of each other,
% but Cor 3.9 is much stronger.

Combining Corollary~\ref{cor:M and HM} and Proposition~\ref{prop:irr},
one can prove

\begin{cor}\label{cor:length irr}
Let $f: M \to N$ be an irreducible map with
$M$ and $N$ non-projective indecomposables in $\stmod{kG}$, and
let $(\cP,\cI)$ be a projective class on $\StMod{kG}$.
If $M$ and $N$ have finite lengths with respect to $(\cP,\cI)$,
then
 \[
   \len_{\cP}(M)-1 \leq \len_{\cP}(N) \leq \len_{\cP}(M)+1 .
 \]
In particular, for $M$ indecomposable and
$S$ any summand of $H(M)$, we have
\[
\len_{\cP}(M)-1 \leq \len_{\cP}(S) \leq \len_{\cP}(M)+1 . \xqedhere{130pt}
\]
\end{cor}

\begin{comment}
Now we define a map $f:M \to N$ to be \dfn{strongly irreducible} if
for any split-epi $p: N \to N'$ and any split-mono $i: M' \to M$,
the map $pfi: M' \to N'$ is irreducible. Then
we get an interesting consequence of Corollary~\ref{cor:length irr}.

\begin{cor}
Let $f: M \to N$ be a strongly irreducible map in $\stmod{kG}$ and
$(\cP,\cI)$ be a projective class on $\StMod{kG}$. If
$M$ and $N$ have finite lengths with respect to $(\cP,\cI)$,
then the lengths of the non-projective summands of $M$
differ by at most $2$, and
the same is true for $N$.
Moreover, if there exist two summands of $M$ whose lengths differ by $2$,
then the lengths of the summands of $N$ are the same. \qed
\end{cor}

Bautista has proved the following relation between irreducible maps and
strongly irreducible maps in $\stmod{kG}$.

\begin{pro}[Bautista~\cite{Bautista}]
A map $f: A \to B$ in $\stmod{kG}$ is irreducible if and only if 
it is a direct sum of a strongly irreducible map and an isomorphism.
\end{pro}

It follows from the definition of a strongly irreducible map that
it is a non-isomorphism on each indecomposable component. Such a map
is said to be in the \dfn{Jacobson radical} of $\StMod{kG}$.
For maps in the Jacobson radical, we have

\begin{pro}[Bautista~\cite{Bautista}]
A map $f$ in $\stmod{kG}$ is strongly irreducible if and only if
it is irreducible and in the Jacobson radical.
\end{pro}

\end{comment}

\subsection{Ghost lengths}\label{ss:gl}

% This short subsection tries to explain the difficulty that we cannot
% apply test objects to evaluate ghost lengths.
% Also works for a general projective class, but then why would
% restricting to thick<k> be relevant.

The results of Sections~\ref{ss:length and tri} to~\ref{ss:irr} apply to the generating
length of a module in $\StMod{kG}$, since generating length is 
the length with respect to the ghost projective class.
When $kG$ has periodic cohomology, there is a projective class on 
$\thick{k}$ whose ideal is $\cG_t$, and ghost length is the length
with respect to this projective class.
In general, we don't know whether ghost length is a length with
respect to a projective class,
but we can still prove the analogue of half of Corollary~\ref{cor:M and HM}:

\begin{pro}
 Let $G$ be a finite group, and let $k$ be a field whose characteristic divides the order of $G$.
 Consider the Auslander-Reiten triangle 
 $\Omega^2M \to H(M) \to M \to \Omega M$ 
 for some indecomposable module $M$ in $\thick k$. 
 Then the following holds:
 \[
   \gl(M) - 1 \leq \gl(H(M))
 \]
\end{pro}

\begin{proof}
 We mimic the proof of Lemma~\ref{le:gl+ar}.
 Suppose that $\gl(H(M)) = l-1$.  We must prove that $\gl(M) \leq l$.
 Since $\gl(M) = \gl(\Omega^2 M)$, it suffices to show
 that any map $\phi : \Omega^2 M \to N$ in $(\cG_t)^l$ is stably trivial,
 where $\cG_t$ consists of ghosts between objects in $\thick{k}$.
 Write $\phi$ as $\phi_2 \phi_1$, where $\phi_1$ is in $\cG_t$
 and $\phi_2$ is in $(\cG_t)^{l-1}$.
 Then, by Proposition~\ref{prop:length r-a-s}, the composite
 $\phi_1 \, \Omega \gamma$ is stably trivial, so $\phi_1$ factors through $H(M)$:
\[
\xymatrix{
 \Omega M \ar[r]^{\Omega \gamma} & \Omega^2 M \ar[r] \ar[d]^{\phi_1} & H(M) \ar[r] \ar@{-->}[dl]^{\psi} & M \ar[r]^{\gamma} & \Omega M \\
                                 & W \ar[d]^{\phi_2} \\
                                 & N .
}
\]
Now since $\gl(H(M)) = l-1$, the composite $\phi_2 \psi$ is stably trivial 
and so $\phi$ is stably trivial as well.
\end{proof}
% Note that the fact that phi_1 \Omega gamma = 0 follows from the formal property 
% of the almost zero map and the fact that a ghost is not split monic or split epic. 
% We can prove this directly without introducing a projective class.
% We could also adjust Lemma~\ref{le:gl+ar} to assume only that gamma has the
% property that it is zero when composed with maps in the appropriate power of I,
% so that it includes this situation too.

The analogue of the other half of Corollary~\ref{cor:M and HM} would say
that $\gl(H(M)) \leq \gl(M) + 1$, and we don't know whether this is true.
% * What about analogues of other results earlier in Section 3?

% I decided to keep the following remark here, since it uses cor:M and HM,
% but to separate it from the discussion of whether ghost length comes from
% a projective class.
\begin{rmk}\label{re:gl=gel}
A related question is whether the generating length and ghost length
always agree.  We know of no counterexamples.
However, Corollary~\ref{cor:M and HM} implies that the longest composite
of ghosts starting from a given module $M$ in $\thick k$ can always be
attained by a map in $(\cG^m)_t$, the intersection of $\cG^m$ and $\thick k$. 
Thus if $(\cG_t)^m=(\cG^m)_t$, then the ghost length and generating length agree.
Note that a related statement for the objects of $\cP$, i.e., 
that $(\cP^c)_n = (\cP_n)^c$, where the superscript $c$ means to take
the intersection with the compact objects,
is known to be true~\cite[2.2.4]{BoVdB}.
% Not the exact analog, since it is for compact objects not objects in thick{k}.
\end{rmk}

\section{Ghost numbers of $p$-groups}\label{se:pgroups}

In this section we study finite $p$-groups, using the fact that $\thick k = \stmod{kG}$.
We begin in Section~\ref{ss:bg} by recalling several results that we will use.  
In Section~\ref{ss:gen-and-socle-length} we show that the generating length invariant
is a stabilized version of the socle length, and give a result that shows that
if these are equal for a module $M$, the same is true for $\rad(M)$ and $M/\soc(M)$.
Then we give new computations of bounds on ghost numbers for various $p$-groups:
abelian $p$-groups in Section~\ref{ss:gn-abelian},
the quaternion group $Q_8$ in Section~\ref{ss:gn-Q8},
dihedral $2$-groups in Section~\ref{ss:gn-D4q},
and the groups $C_{p^r} \times C_{p^s}$ in Section~\ref{ss:gn-C3xC3}.
In several cases we determine the ghost number completely, such as for
$D_{4q}$, $C_3 \times C_{3^s}$ and $C_4 \times C_{2^s}$.
In Section~\ref{ss:cyclic sub}, we compute the ghost length and generating length
of certain modules induced up from a cyclic normal subgroup.  
This is used in the same section to show that the ghost number and the 
radical length are within a factor of three of each other for any $p$-group.
It is also used in Section~\ref{ss:gn-D4q} in the computation of the
ghost number of $kD_{4q}$ and in Section~\ref{ss:possible}, where we
classify group algebras with small ghost number and put constraints
on which ghost numbers can occur.

When we write ``$p$-group'', we always mean ``finite $p$-group''.

\subsection{Background}\label{ss:bg}

We recall the following theorem, 
and then explain the terminology and give an idea of the proof.

\begin{thm}[Chebolu, Christensen and Min\'a\v{c}~\cite{Gh in rep}]\label{th:finite gl}
Let $G$ be a $p$-group, and let $k$ be a field of characteristic $p$. 
Then the generating length of a $kG$-module $M$ is at most its radical length,
and the following inequalities hold:
\[
\text{ghost number of } kG \leq \text{generating number of } kG 
                              <      \text{nilpotency index of } J(kG) \leq |G| .
\]
In particular, the ghost number of $kG$ is finite in this case.\qed
\end{thm}

Let $G$ be any finite group, and let $k$ be a field whose characteristic divides the order of $G$.
Let $J=J(kG)$ be the Jacobson radical of $kG$, i.e., the largest nilpotent ideal of $kG$. 
The nilpotency index of $J(kG)$ is the smallest integer $m$ such that $J^m=0$, 
and for any module $M$, we have a radical series
\[ M=\rad^0(M) \supseteq \rad^1(M) \supseteq \rad^2(M) \supseteq \cdots \supseteq 0, \]
with $\rad^n(M)=J^n M $, and a socle series
\[ 0=\soc^0(M) \subseteq \soc^1(M) \subseteq \soc^2(M) \subseteq \cdots \subseteq M, \]
with $\soc^n(M)$ consisting of the elements of $M$ annihilated by $J^n$. 
The radical length of $M$ is the smallest integer $n$ such that $\rad^n(M)=0$. 
This is equal to the socle length of $M$,
the smallest integer $m$ such that $\soc^m(M)=M$. 
The successive quotients in the sequences are direct sums of simple modules. 

If $G$ is a $p$-group, then each quotient is a direct sum of $k$'s,
so the generating length of a module $M$ is less than or equal to its radical length.
Note that the nilpotency index of $J(kG)$ is exactly the radical length of $kG$,
and if $M$ is a projective-free $kG$-module, it always has smaller radical length than $kG$.
The theorem then follows.

The following lemma is proved by studying Tate cohomology in degrees $0$ and $-1$. 
We write $\rad(M)$ for $\rad^1(M)$ and $\soc(M)$ for $\soc^1(M)$.

\begin{lem}[Chebolu, Christensen and Min\'{a}\v{c}~\cite{Gh in rep}]\label{le:soc}
Let $G$ be a $p$-group, and let $k$ be a field of characteristic $p$. 
Let $f: M \to N$ be a map in $\Mod{kG}$ between projective-free modules $M$ and $N$. 
Then:
\begin{enumerate}[(a)]
\item $\soc(M)\subseteq \ker(f)$ iff $[k, f] = 0$.
\item $\im(f) \subseteq \rad(N)$ iff $[\Omega^{-1} k, f] = 0$.
\end{enumerate}
In particular, if $f$ represents a ghost in the stable category, then
both inclusions hold.\qed
\end{lem}

As a corollary, we get

\begin{cor}[Chebolu, Christensen and Min\'{a}\v{c}~\cite{Gh in rep}]\label{co:soc}
Let $G$ be a $p$-group, and let $k$ be a field of characteristic $p$. 
Let $f: M \to N$ be a map in $\Mod{kG}$ between projective-free modules $M$ and $N$. 
If $f$ is an $l$-fold ghost, then:
\begin{enumerate}[(a)]
\item $\soc^l(M)\subseteq \ker(f)$.
\item $\im(f) \subseteq \rad^l(N)$.\qed
\end{enumerate}
\end{cor}

The next lemma provides ghosts with a particular form.

\begin{lem}[Benson, Chebolu, Christensen and Min\'a\v{c}~\cite{GH for p}]\label{le:x-1}
Let $G$ be a $p$-group, and let $k$ be a field of characteristic $p$. 
Let $x \in G$ be a central element.  
Then left multiplication by $x-1$ on a $kG$-module $M$ is a ghost.\qed
\end{lem}

Note that in general there are ghosts not of this form.
Nevertheless these ghosts work well for abelian groups
in providing lower bounds for ghost numbers (see Section~\ref{ss:gn-abelian}).
It is not hard to check that if $G$ is a cyclic $p$-group with generator $g$, 
then $g-1$ is a universal ghost.

\subsection{Generating and socle lengths}\label{ss:gen-and-socle-length}

We now show that
the generating length is a stabilized version of the socle length.
In this section we allow our projective classes to be unstable,
that is, we don't assume that the projectives are closed under
suspension and desuspension.

Let $G$ be a $p$-group, let $k$ be a field of characteristic $p$,
and let $M$ be a $kG$-module.
Note that $\soc(M)$ contains exactly the image of maps from $k$.
So, when we build up $M$ in a socle sequence in Theorem~\ref{th:finite gl},
we are only using maps from $k$, not all suspensions of $k$.
This suggests that we consider the unstable projective class generated by $k$
in $\StMod{kG}$.  We will show that the length with respect to this projective class
is exactly the socle length for projective-free modules in $\stmod{kG}$.

% Note that the unstable class restricts to $\stmod{kG}$.

Note that the regular representation $kG$ is the only indecomposable projective $kG$-module,
and $\soc(kG) \iso k$ is its unique minimal left submodule.
Thus any map $kG \to M$ in $\Mod{kG}$ with $M$ projective-free has
$\soc(kG)$ in its kernel,
since the map cannot be injective.
It follows that a map $\oplus k \to M$ in $\Mod{kG}$ with $M$ projective-free
is stably trivial if and only if it is the zero map.
For finitely generated modules, 
a similar argument shows that the same is true for a map $M \to \oplus k$ in $\modu{kG}$ with $M$ projective-free.

\begin{pro}
Let $G$ be a $p$-group, and let $k$ be a field of characteristic $p$.
Let $(\cP, \cI)$ be the unstable projective class in $\StMod{kG}$ generated by $k$.
Then a map $f: M \to N$ between projective-free objects $M$ and $N$ is in $\cI$
if and only if it is represented by a map $f$
such that $\soc(M) \subseteq \ker(f)$.
Hence, if $M$ is finitely-generated and projective-free,
the length of $M$ with respect to $(\cP, \cI)$ is exactly its socle length.
\end{pro}

\begin{proof}
That $f \in \cI$ is equivalent to $\soc(M) \subseteq \ker(f)$ is Lemma~\ref{le:soc}~(a).

Now let $M$ be projective-free. Then $M \to M / \soc(M)$ is a
universal map in $\cI$. It follows that $M \to M/ \soc^{k}(M)$
is universal in $\cI^{k}$.
If $M$ has socle length $n$, then $M \in \cP^n$ and
$M \to M/ \soc^{n-1}(M)$ is non-zero.
If further $M$ is finitely-generated, then the universal map
$M \to M/ \soc^{n-1}(M) \iso \oplus k$ is stably non-trivial,
by the remarks preceding this proposition.
Thus $M$ has length $n$ with respect to $(\cP, \cI)$.
\end{proof}

Note that the stable projective class generated by $k$ in $\StMod{kG}$ is exactly the ghost projective class.
Thus the generating length is indeed the socle length stabilized and
is generally less than or equal to the socle length.
We have also recovered Theorem~\ref{th:finite gl} from this observation.
In Section~\ref{ss:cyclic sub}, we are going to prove that
the generating number of $kG$ is within a factor of 3 of the socle length of $kG$.

Here we show that if the generating length of a module $M \in \StMod{kG}$
happens to equal its socle length 
(see, for example, Proposition~\ref{prop:abelian module} and Theorem~\ref{th:rl=gel}),
then the same holds for $\rad(M)$ and $M/\soc(M)$,
a result that we will use in Section~\ref{ss:gn-D4q} when studying dihedral groups.

\begin{pro}\label{prop:length of rad}
Let $k$ be a field of characteristic $p$, and let $G$ be a $p$-group.
Assume that $M \in \StMod{kG}$ has generating length equal to its radical length.
Then $\gel(M/ \soc(M))=\gel(M)-1$, and similarly $\gel(\rad(M)) = \gel(M)-1$.
\end{pro}

\begin{proof}
Since the generating length of $M$ is strictly less than the nilpotency index of $J(kG)$, $M$ is projective-free.
The proposition is then a special case of the following more general lemma.
\end{proof}

% Note that we are using the fact that ghosts and dual ghosts coincide when we say that
% the result on $\rad(M)$ is dual to that on $M/ \soc(M)$.
% However, since radical lengths and socle lengths are equal, we don't really need
% an injective class here.

% I might have meant ghost length here, since we only know that
% ghosts and dual ghosts coincide in $\stmod{kG}$, not the big category.
% So ghost length should equal dual ghost length, but not generating length
% and dual generating length.

\begin{lem}
Let $\T$ be a triangulated category, and let $(\cP,\cI)$ and $(\cP',\cI')$ be (possibly unstable) projective classes
on $\T$ such that $\cP' \subseteq \cP$.
Suppose that $M \in \T$ has $\len_{\cP'}(M)= \len_{\cP}(M)=m$ and that there exist
$L \in \cP'_{m-n}$ and $N \in \cP'_n$ with a triangle
\[ L \to M \to N. \]
Then
\[ \len_{\cP'}(L)= \len_{\cP}(L)=m-n, \text{ and }
 \len_{\cP'}(N)= \len_{\cP}(N)=n.\]
\end{lem}

\begin{proof}
We have that $\len_{\cP'}(L) \leq m-n$ and $\len_{\cP'}(N) \leq n$.
But $\len_{\cP'}(L)+ \len_{\cP'}(N) \geq m = (m-n) + n$, so the equalities follow
for $(\cP',\cI')$.
Since $\cP' \subseteq \cP$, the same results hold for $(\cP, \cI)$ too.
\end{proof}

Intuitively, this easy fact says that when $\len_{\cP'}(M)= \len_{\cP}(M)$,
the related object $L$ can be built from $\cP'$ as efficiently as it can
be built from $\cP$.
It applies to generating lengths and socle lengths.

%We also remark here that,
%when we defined the length of $M$ with respect to a projective class $(\cP, \cI)$, 
%we said that $M$ is in $\cP_n$ if and only if it is a retract of an object $N$ that
%sits inside a cofibre sequence $P \to N \to Q$ with $P \in \cP$ and 
%$Q \in \cP_{n-1}$,
%so that it works for unstable projective classes.
%When the projective class is stable, it is equivalent to require
%that $N$ be the fibre or cofibre of a map $Q \to P$ with $P \in \cP$ and 
%$Q \in \cP_{n-1}$.

\medskip
We now provide examples of computations of ghost numbers of certain groups,
improving on results in~\cite{Gh in rep}.

\subsection{Ghost numbers of abelian $p$-groups}\label{ss:gn-abelian}

We first prove a general proposition. It generalizes~\cite[Lemma 2.3]{GH for p}
and~\cite[Prop.~5.10]{Gh in rep}.

\begin{pro}\label{pr:induced-from-cyclic}
Let $k$ be a field of characteristic $p$, and let $H$ be a non-trivial subgroup of a $p$-group $G$. 
Assume that there exists a central element $x$ in $G$.
Let $l$ be the smallest positive integer such that $x^l \in H$.
Suppose that $M \in \StMod{kH}$ has generating length $m \geq 1$. Then
$ \gel(M \up ^G) \geq \gel(M) + (l-1)$, and
\[\text{generating number of }kG \geq \text{generating number of }kH +(l-1).\]
Suppose that $M \in \stmod{kH}$ has ghost length $n \geq 1$. Then
$ \gl(M \up ^G) \geq \gl(M) + (l-1)$, and
\[\text{ghost number of }kG \geq \text{ghost number of }kH +(l-1).\]
\end{pro}
% The p-group assumption is only being used to ensure that the
% induced map is in the thick subcategory generated by k.

% Technically, the ``non-trivial'' assumption on H can be dropped, since
% it is implied by the assumption that m \geq 1.

% The special case when $H=G$ is that $x-1$ is a ghost.  jdc:  Really?
% When H=G, then l=1, so we only consider (x-1)^0 = 1.

% The special case when $l=1$ is that ghosts induce up to ghosts.
% (E.g. this happens when x = 1.)

\begin{proof}
For brevity, we write $\down$ for $\down^G_H$ and $\up$ for $\up^G_H$.
Let $f: M \to N$ be a non-trivial $(m-1)$-fold ghost in $\StMod{kH}$.
We will show that $(x-1)^{l-1} \circ f \up$ is stably non-trivial.
Since ghosts induce up to ghosts and $x-1$ is a ghost, it
follows that there exists a non-trivial composite of $(m-1) + (l-1)$
ghosts in $\StMod{kG}$.

Consider the map $M \xrightarrow{i} M\up\down \xrightarrow{f\up\down}
          N\up\down \xrightarrow{(x-1)^{l-1} \down} N\up\down \xrightarrow{r} N$,
where $i$ and $r$ are the natural maps.
To be more explicit, $M\up_H^G = kG \otimes_H M$,
$i(\alpha) = 1 \otimes \alpha$ and
$r(g \otimes \alpha) = g \alpha$ if $g \in H$ and is zero otherwise.
By naturality of the inclusion, the composite equals
$M \xrightarrow{f} N \xrightarrow{i} N\up\down \xrightarrow{(x-1)^{l-1} \down}
N\up\down \xrightarrow{r} N$.
Since $x^i \not \in H$ for $i \leq l-1$, the map $N \xrightarrow{i} N\up\down \xrightarrow{(x-1)^{l-1} \down}
N\up\down \xrightarrow{r} N$ is simply multiplication by $(-1)^{l-1}$, an isomorphism.
Since $N$ is stably non-zero, it follows that $(x-1)^{l-1}\down \circ f \up\down$ and therefore
$(x-1)^{l-1} \circ f \up$ are stably non-trivial.

The result on ghost length and ghost number can be proved similarly by replacing $\StMod{kG}$ with $\stmod{kG}$.
\end{proof}

We can apply this proposition to abelian groups.

\begin{thm}\label{th:ab}
Let $k$ be a field of characteristic $p$, and
let $A=C_{p^r}\times C_{p^{r_1}}\times \cdots \times C_{p^{r_l}}$ be an abelian $p$-group. 
Then
\[ m-p^r+\left\lceil \frac{p^r-1}{2}\right\rceil \leq \text{ghost number of } kA
   \leq \text{generating number of } kA \leq m-1, \]
where $m$ is the nilpotency index of $J(kA)$, 
and $p^r$ is the order of the smallest cyclic summand.
\end{thm}

When the prime $p$ is greater than 2,
the result here improves on that in~\cite{Gh in rep}, 
where the lower bound for the ghost number of $kA$ is given by $m-p^r+p^{r-1}
= m - p^r + \lceil (p^r - 1)/p \rceil$.

Note that since
\[ m=1+(p^r-1)+(p^{r_1}-1)+\cdots+(p^{r_l}-1), \]
our lower bound can also be written as
\[ \left\lceil \frac{p^r-1}{2}\right\rceil + (p^{r_1}-1)+\cdots+(p^{r_l}-1). \]

% This comment is a bit redundant, since I used the result 
% on cyclic groups in the proof. But it does show that
% we are getting a better lower bound. 
Also note that when $A$ is cyclic, we have $m=p^r$, and the lower bound 
$d = \lceil \frac{p^r-1}{2}\rceil$ here is exactly
the ghost number of $A$~\cite[Thm.~5.4]{Gh in rep}.

\begin{proof}
Let $g$ be a generator of $C_{p^r}$, and let
$g_i$ be a generator of $C_{p^{r_i}}$, $i=1,2,\ldots,l$. 
Write $d = \big\lceil \frac{p^r-1}{2}\big\rceil$.
By the proof of~\cite[Prop.~5.3]{Gh in rep}, $kC_{p^r}$ has ghost number $d$.
We can now apply Proposition~\ref{pr:induced-from-cyclic} by successively including
the summands $C_{p^{r_i}}$ to obtain
\[\text{ghost number of $kA$} \geq
 d + (p^{r_1}-1)+\cdots+(p^{r_l}-1). \]
The other inequalities are from Theorem~\ref{th:finite gl}.
\end{proof}

Proposition~\ref{pr:induced-from-cyclic} allows us to make this explicit.
Let $M=N \up_{C_{p^r}}^A$, with $N=kC_{p^r}/(g-1)^d$. 
Note that $(g-1)^{d-1}$ is a stably non-trivial
$(d-1)$-fold ghost on $N$ in $\stmod{kC_{p^r}}$ and, since $A$ is abelian,
the self map $(g-1)\up_{C_{p^r}}^A$ on $M$ is simply
left multiplication by $g-1$. Hence we have a particular form
for the non-trivial $(m-p^r+d-1)$-fold ghost on $M$:
\[ \theta=(g-1)^{d-1}(g_1-1)^{p^{r_1}-1}\cdots(g_l-1)^{p^{r_l}-1}. \]

More generally, we have the following result.

\begin{pro}\label{prop:abelian module}
Let $k$ be a field of characteristic $p$,
let $A=C_{p^{r_1}}\times C_{p^{r_2}}\times \cdots \times C_{p^{r_l}}$ be an abelian $p$-group, and
let $M_i$ be an indecomposable $C_{p^{r_i}}$-module of dimension $n_i$ for each $i$.
Then the $A$-module $M=M_1 \otimes \cdots \otimes M_l$ has radical length
$1 + (n_1-1) + \cdots + (n_l-1)$.
If $n_i \leq \frac{p^{r_i}}{2}$ for some $i$, then the generating length of $M$ equals its radical length.
\end{pro}

Before proving the proposition, we state the following lemma.

\begin{lem}[{\cite[Theorem~1.2]{Jennings}}]
Let $G$ be a $p$-group, and let $k$ be a field of characteristic $p$.
Then the elements $h-1$ with $h \neq 1$ form a basis for $\rad(kG)$.
It follows that the products $(h_1-1) \cdots (h_n-1)$ with $h_i \neq 1$
span $\rad^n(kG)$. \qed
% Using h, h_i not g, g_i since g, g_i are specific elements above.
\end{lem}

Note that it suffices to consider generators of the group $G$
when we generate $\rad^n{kG}$ as a sub-module.
We can now compute the radical length of the module $M$ and prove the proposition.

\begin{proof}[Proof of Proposition]
Let $g_i$ be a generator of $C_{p^{r_i}}$.
Then the various $g_i-1$ with $1 \leq i \leq l$ generate $\rad(kG)$.
We regard $M_i$ as the quotient $kC_{p^{r_i}}/(g_i-1)^{n_i}$, so
the elements $(g_i-1)^j$ with $0 \leq j \leq n_i-1$ form a basis of $M_i$.
Now let $m=(n_1-1) + \cdots + (n_l-1)$.
Since any $(m+1)$-fold product of the elements $g_i-1$ has to be zero in $M$,
$\rad^{m+1}(M)=0$.
On the other hand, the element
$(g_1-1)^{n_1-1} \otimes \cdots \otimes (g_l-1)^{n_l-1} \in M$
is non-zero and spans $\rad^m (M)$.
It follows that the radical length of $M$ is $m+1$.

To prove the last statement,
without loss of generality we can assume that $n_1 \leq \frac{p^{r_1}}{2}$.
We then consider the restriction of $M$ to $H=C_{p^{r_1}}$.
Note that we have a vector space isomorphism
\[M\down_H \iso \bigoplus_{i_2=0}^{n_2-1} \cdots \bigoplus_{i_l=0}^{n_l-1} M_1.\]
Since $G$ acts componentwise, this is actually an isomorphism of $kH$-modules,
and we have $kH$-maps $i: M_1 \to M\down_H$ sending $\alpha$ to
$\alpha \otimes 1 \otimes \cdots \otimes 1$ and
$r: M\down_H \to M_1$ sending $\alpha \otimes (g_2-1)^{i_2} \otimes \cdots \otimes (g_l-1)^{i_l}$ to $(-1)^{i_2 + \cdots + i_l}\alpha$ for $0 \leq i_k \leq n_k - 1$.

We can form the $m$-fold ghost $f=(g_1-1)^{n_1-1} \cdots (g_l-1)^{n_l-1}$ on $M$.
And one can check that $r \circ f\down_H \circ i$ is $\pm (g_1-1)^{n_1-1}$ on $M_1$,
which is stably non-trivial.
Hence $f$ is stably non-trivial and the ghost length of $M$ is at least $m+1$.
Since this is also the radical length of $M$, we have
$\gl(M) = \gel(M) = m+1$.
\end{proof}

\begin{rmk}
We don't know which of the lower bound and upper bound better approximates
the ghost number in general,
but we suspect that the lower bound is better.
We show in Section~\ref{ss:gn-C3xC3} that the upper bound can be refined by $1$
for rank $2$ abelian $p$-groups $C_{p^r} \times C_{p^s}$,
with $p^r, p^s \geq 3$.
In particular, the lower bound we have here is the exact ghost number
for the group $C_3 \times C_3$. %and many others
\end{rmk}

\subsection{Ghost number of the quaternion group $Q_8$}\label{ss:gn-Q8}

In this section, we study the quaternion group
$Q_8 = \langle \epsilon, i,j \st \epsilon^2=1,\, i^2=j^2=(ij)^2=\epsilon \rangle$
over a field $k$ of characteristic 2.  
It has been shown in~\cite{Gh in rep} that 
the ghost number of $kQ_8$ is $2$, $3$, or $4$.

\begin{pro}\label{prop:Q_8}
Let $k$ be a field of characteristic $2$. 
Then there exists a stably non-trivial double ghost in $\stmod{kQ_8}$. 
Hence
\[
3 \leq \text{ghost number of } kQ_8 \leq \text{generating number of } kQ_8 \leq 4.
\] 
\end{pro}

\begin{proof}
We have a quotient map from $Q_8$ to the Klein four group $V$
that identifies $\epsilon$ with $1$. 
We also write $i$ and $j$ for the generators of $V$.
The rank one free $kV$-module can be viewed as a $kQ_8$-module,
and we write $kV$ for it. It has radical length $3$, and we will
show that it admits a stably non-trivial double ghost, hence
$\gl(kV)=\gel(kV)=3$.

Right multiplication $R_{i+1}$ on $kV$ by $i+1$ is a left $kQ_8$-map, and
we claim that it is a ghost.
To see this, consider the short exact sequence
\[ 0 \to kV \xrightarrow{\iota} kQ_8 \to kV \to 0 \]
of left $kQ_8$-modules,
where the kernel $kV$ is generated by $\epsilon+1$ in $kQ_8$. 
It follows from this sequence that $\Omega kV = kV$ and
that $\Omega R_{i+1} = R_{i+1}$.

Thus to show that $R_{i+1}$ is a ghost, we just need to check that 
it is stably trivial on maps from $k$.
Multiplication by $i+1$ kills the socle of $kV$, which is generated by $1+i+j+ij$,
so this follows from Lemma~\ref{le:soc}(a).

Next we show that there is a non-trivial double ghost.
For any map $f : kQ_8 \to kV$, the composite $f \iota$ is zero, since $\epsilon+1$
acts trivially on $kV$.
Thus a $kQ_8$-map $kV \to kV$ is stably trivial if and only if it is zero.
As a result, multiplication by $(i+1)(j+1)$ on $kV$ is stably non-trivial, 
and we get the desired double ghost.

It follows that the ghost number of $kQ_8$ is at least $3$. 
The nilpotency index of $J(kQ_8)$ is $5$, so the generating number of $kQ_8$ is at most $4$.
\end{proof}

\begin{rmk}\label{rmk:Q_8}
The map $R_{(i+1)(j+1)} = R_{1 + i + j + ij} : kV \to kV$
constructed in the proof is in fact the almost zero map with domain
$kV$ in $\stmod{kQ_8}$.
To see this, we consider the inclusion $\rad(kV) \to kV$.
Since this map is not split-epi, its composition with the almost zero map
$\gamma: kV \to kV$ factors through a projective module $P$.
But $P$ is also injective, thus we can change $\gamma$ by a map factoring through $P$
to ensure that $\rad(kV) \subseteq \ker(\gamma)$.
Since $kV/ \rad(kV) \iso \soc(kV) \iso k$ and $\soc(kV)$ is generated
by the element $1+i+j+ij$, it must be
that $R_{1 + i + j + ij}$ is the almost zero map (up to a scalar factor).
This gives another proof that this map is stably non-trivial.
%
% In general, if M has socle length m, then the composite
% soc^{m-1}(M) --> M --> \Omega M must be stably trivial,
% and so the almost zero map factors via a map
% M/soc^{m-1}(M) --> \Omega M.
% In our case (M=kV), M/soc^{m-1}(M) = k, and there is just a 1d space
% of such maps which includes 1 + i + j + ij.

% The inclusion rad(M) --> soc^{m-1}(M) holds in general,
% and they coincide in our case.
\end{rmk}

In the next section, we generalize the technique used here.

\subsection{$p$-groups with cyclic normal subgroups}\label{ss:cyclic sub}

In Section~\ref{ss:gn-abelian}, we produced ghosts using left multiplication
by $x-1$ for abelian groups.
More generally, in Lemma~\ref{le:x-1}, we saw that left multiplication by
$x-1$ for $x$ a central element produces a ghost.
For a non-central element, in order to produce a left module map, one
must consider \emph{right} multiplication, when this makes sense,
and indeed we used this technique in Section~\ref{ss:gn-Q8} to produce
ghosts for $Q_8$.
However, it is not always true that right multiplication by $x-1$ produces ghosts.
Generalizing the known examples, we show that if $M$ is induced up
from a cyclic normal subgroup, then right multiplication by $x-1$ on $M$ is
well-defined and is a ghost.

\begin{thm}\label{th:rl=gel}
Let $C_{p^r}$ be a cyclic normal subgroup of a $p$-group $G$, and let $k$ be a field of characteristic $p$.
Let $M_n$ be an indecomposable $kC_{p^r}$-module of dimension $n$, and write $M = M_n\up^G$.
Then, for each $x \in G$, one can define the right multiplication map $R_{x-1}$ on $M$ and it is a ghost.
Moreover, if $n \leq \lceil \frac{p^r-1}{2} \rceil$, then $\gl(M)=\gel(M)=\rl M$.
\end{thm}

Note that for $n=1$, we have $M \iso kH \down_G$, where $H=G/C_{p^r}$ and the 
restriction is taken along the quotient map.
Thus the ghosts in the previous section are examples of this construction.

\begin{proof}
Let $g$ be a generator of $C_{p^r}$.
We can identify $M_n$ with the left submodule of $kC_{p^r}$ generated by $(g-1)^{p^r-n}$,
and so we have a short exact sequence of $kC_{p^r}$-modules:
\[ 0 \to M_n \to kC_{p^r} \to M_{p^r-n} \to 0 , \]
where $M_{p^r-n}$ is an indecomposable $kC_{p^r}$-module of dimension $p^r-n$.
Inducing up, we get
\begin{equation}\label{eq:lower bound} 
0 \to M_n\up^G \xrightarrow{\ i\ } kG \xrightarrow{\ p\ } M_{p^r-n} \up^G \to 0.
\end{equation}
The inclusion $i$ identifies $M = M_n\up^G$ with the left submodule of $kG$
generated by $(g-1)^{p^r-n}$.
Since $C_{p^r} \leq G$ is normal, this submodule is actually a sub-bimodule.
Thus the right multiplication map $R_{x-1} : M \to M$ is well-defined and is a 
left $kG$-module map, for each $x \in G$.
We must show that it is a ghost.

Since~\eqref{eq:lower bound} is in fact a short exact sequence of bimodules,
$R_{x-1}$ is two-periodic as a left $kG$-map, so
it suffices to check that $R_{x-1}$ is left stably-trivial on maps from
$k$ and $\Omega^{-1} k$. 
By Lemma~\ref{le:soc}, this is equivalent to
$\soc_L(M) \subseteq \ker(R_{x-1})$ and
$\im(R_{x-1}) \subseteq \rad_L(M)$, 
where we use subscripts to indicate left and right socles and radicals.
Clearly, $\soc_R(M) \subseteq \ker(R_{x-1})$ and
$\im(R_{x-1}) \subseteq \rad_R(M)$.
Now $\soc_L(kG)=\soc_R(kG) \iso k$, so $\soc_L(M)=\soc_R(M) \iso k$,
which gives the first inclusion.
And one can also show that $\rad_L(M)=\rad_R(M)$, which gives the second
inclusion.

To prove the last claim, let $n \leq \lceil \frac{p^r-1}{2} \rceil$ and assume that $\rl M =l$. 
We want to construct an $(l-1)$-fold ghost.
Note that $\soc_L(M)=\soc_R(M)=\rad_R^{l-1}(M)=M(g_1-1)\cdots(g_{l-1}-1)$
for some $g_1, \ldots, g_{l-1}$ in $G$, so
the $(l-1)$-fold ghost $f := R_{g_{l-1}-1}\circ \cdots \circ R_{g_1-1}$ takes $M$ onto its socle.
For any map $h : kG \to M$, the composite $h i$ is zero, since
the image of $i$ is generated by $(g-1)^{p^r-n}$ which acts trivially on $M$ since $n \leq p^r - n$.
Thus a map $M \to M$ is stably trivial if and only if it is zero,
and so our $(l-1)$-fold ghost $f$ is stably non-trivial.
Thus $l \leq \gl(M) \leq \gel(M) \leq \rl(M) = l$, and we are done.
\end{proof}

\begin{rmk}
As in Remark~\ref{rmk:Q_8}, we can also see that $f$ is non-trivial using
the theory of Auslander-Reiten triangles.
There is a canonical inclusion $j$ of $M$ into $M_{p^r-n}\up^G = \Omega M$ induced from
the $kC_{p^r}$-map $M_n \to M_{p^r-n}$, 
and one can show that the composite $j f$ is exactly the almost zero map out of $M$.
% The almost zero map can be chosen to vanish on rad^{l-1}(M), and so factors
% through the one-dimensional trivial module M/rad^{l-1}(M), and so lands in
% the socle of M_{p^r-n}\up^G.
\end{rmk}

Note that any $p$-group $G$ has a non-trivial center,
hence a cyclic normal subgroup $C_p$. Applying the theorem to
the short exact sequence of groups $C_p \to G \to H$, we get

\begin{cor}\label{co:lower-bound}
Let $G$ be a $p$-group, and let $k$ be a field of characteristic $p$. Then
\[ \frac{1}{2} \rl kG \leq \gn kG \leq \gen kG < \rl kG, \]
when $p$ is even, and
\[ \frac{1}{3} \rl kG \leq \gn kG \leq \gen kG < \rl kG, \]
when $p$ is odd.
\end{cor}

\begin{proof}
Choose a cyclic normal subgroup $C_p$ of $G$, and let $M = M_n\up^G$,
where $M_n$ is an indecomposable $kC_p$-module of dimension
$n = \lceil \frac{p-1}{2} \rceil$.
Since $\rl M = \gl(M) \leq \gn kG$, we only need to show that $2(\rl M) \geq \rl kG$
for $p$ even and $3(\rl M) \geq \rl kG$ for $p$ odd.
By~\eqref{eq:lower bound}, we know that
$\rl M + \rl M_{p-n}\up^G \geq \rl kG$.

For $p$ even, $p - n = n$, and so the result follows.

For $p$ odd, $p - n = n+1$.  
We will show that $2(\rl M) \geq \rl M_{n+1}\up^G$, and the corollary will follow.
There is a short exact sequence
\[0 \to M \to M_{n+1}\up^G \to M_1 \up^G \to 0,\]
induced up from $C_p$-maps,
and one sees that $M_1 \up^G$ is a submodule of $M$ again by inducing up
the $C_p$-map $k \to M_n$.
It follows that $2(\rl M) \geq \rl M + \rl M_1 \up^G \geq \rl M_{n+1}\up^G$.
\end{proof}

We expect that for odd primes, the lower bound can be improved to
an expression that is generically close to $(\rl kG)/2$.
% Maybe when p is large?

% Note that we get a factor of $1/p$ in the lower bound if we choose $n=1$ and
% $M_1 \up^G=kH\down_G$, with $H=G/C_p$, i.e.,
% \[p(\rl kH\down_G) \geq \rl kG.\]

% Also note that for G=C_9, kG has radical length 9 and M = M_1\up^G has 
% ghost length = radical length = 3, so we can't get better than a
% factor of 3 from the above proof.
%    For this example, we can take the cyclic subgroup to be G itself, and M = M_4.
%    Actually, for a cyclic group, the factor (p^r-1)/2p^r is sharp.
%    Don't know how to generalise this to arbitrary p-groups.
%    On the other hand, we know the factor 1/3 is sharp by looking into the group C_3.

\subsection{Ghost numbers of dihedral $2$-groups}\label{ss:gn-D4q}

Our next goal is to study the dihedral $2$-groups.
We will show that the ghost number and generating number of $kD_{4q}$
are both $q+1$.
Here we write $D_{4q}$ for the dihedral $2$-group of order $4q$, with $q$ a power of $2$:
\[ D_{4q} = \langle x,y \st x^2=y^2=1,(xy)^q=(yx)^q \rangle . \]
It has a normal cyclic subgroup $C_{2q}$, generated by $g = xy$. 

Since $kC_{2q}$ has ghost number $q$, which is realized by
the ghost length of $M = kC_{2q}/(g-1)^q$~\cite[Prop.~5.3]{Gh in rep},
the ghost length of $N=M \up_{C_{2q}}^{D_{4q}}$ is at least $q$ in $\stmod{kD_{4q}}$. 
By Theorem~\ref{th:rl=gel}, we actually have $\gl(N) = \gel(N) = \rl N$.
Note that $(xy)^q\in D_{4q}$ is central of order $2$ and that
$M \iso k \up_{C_2}^{C_{2q}}$, 
hence $N = M \up_{C_{2q}}^{D_{4q}} \iso k \up_{C_2}^{D_{4q}}
\iso kD_{2q}\down ^{D_{2q}}_{D_{4q}}$,
where the restriction is along the quotient map in the short exact sequence
$C_2 \to D_{4q} \to D_{2q}$.
It is not hard to see that the radical length of $kD_{2q}$ is $q+1$ 
(see Remark~\ref{rmk:D_n})
and that its $q$-th radical is generated by $((y-1)(x-1))^{\frac{q}{2}}=
((x-1)(y-1))^{\frac{q}{2}}$
(which makes sense for $q=1$ since we have identified $x=y$ in that case).
Thus we have proved the following consequence of Theorem~\ref{th:rl=gel}:

\begin{cor}\label{cor:D_4q lower bound}
Let $k$ be a field of characteristic $2$. 
Then the ghost number of $kD_{4q}$ is at least $q+1$.
In fact, $\gl(N) = \gel(N) = q+1$, where $N = k \up_{C_2}^{D_{4q}}$.\qed
\end{cor}

The proof of Theorem~\ref{th:rl=gel} shows that an explicit $q$-fold
ghost $N \to N$ is given by the right multiplication map $R_{((x-1)(y-1))^{\frac{q}{2}}}$.
\medskip

To get upper bounds for the generating numbers of dihedral $2$-groups, 
we need classification theorems~\cite{Benson}.

Let $\Lambda=k\langle X,Y \rangle / (X^2,Y^2)$ be the quotient of 
the free algebra on two non-commuting variables. %$X$ and $Y$ by 
%the ideal generated by $X^2$ and $Y^2$, over $k$. 
In $kD_{4q}$, writing $X = x-1$ and $Y = y-1$, one can show that
$(XY)^r - (YX)^r = (xy)^r - (yx)^r$ for $r$ a power of $2$, and so
$kD_{4q}\iso \Lambda/((XY)^q-(YX)^q)$~\cite[Lemma~4.11.1]{Benson}.

In the isomorphism $kD_{4q}\iso \Lambda/((XY)^q-(YX)^q)$, we have implicitly
assumed that the characteristic of $k$ is $2$.
However, for the classification we describe below, $k$ can have 
any characteristic, 
and we apply it in this generality in the next section.

$\Lambda$-modules are classifiable. 
Let $\cW$ be the set of words in the \dfn{direct letters} $a$ and $b$ 
and the \dfn{inverse letters} $a^{-1}$ and $b^{-1}$, 
such that $a$ and $a^{-1}$ are always followed by $b$ or $b^{-1}$ and vice versa, 
together with the ``\dfn{zero length word}'' $1$.

Given $C=l_1\cdots l_n\in \cW$, where each $l_i$ is a direct or inverse letter,
let $M(C)$ be the vector space over $k$ with basis 
$z_0, \ldots, z_n$ on which $\Lambda$ acts according to the schema
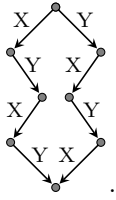
\[ kz_0 \xleftarrow{l_1} kz_1 \xleftarrow{l_2} kz_2 \cdots kz_{n-1} \xleftarrow{l_n} kz_n , \]
with $X$ acting via $a$ and $Y$ acting via $b$. 
For example, if $C = ab^{-1}a^{-1}$, then the schema is 
\[ kz_0 \xleftarrow{a} kz_1 \xrightarrow{b} kz_2 \xrightarrow{a} kz_3 \]
and the module $M(ab^{-1}a^{-1})$ is given by
\[
X \mapsto 
\begin{bmatrix}
0 & 0 & 0 & 0\\
1 & 0 & 0 & 0\\
0 & 0 & 0 & 1\\
0 & 0 & 0 & 0
\end{bmatrix}
\quad\text{ and }\quad
Y \mapsto 
\begin{bmatrix}
0 & 0 & 0 & 0\\
0 & 0 & 1 & 0\\
0 & 0 & 0 & 0\\
0 & 0 & 0 & 0
\end{bmatrix}
\]
with the matrices acting on row vectors on the right. 
Such a module is called a \dfn{module of the first kind}. 
Clearly, $M(C)\iso M(C^{-1})$, where $C^{-1}$ reverses the order of the letters in $C$ 
and inverts each letter.

Let $C=l_1\cdots l_n$ be a word in $\cW$ of even non-zero length that is 
not a power of a smaller word, and let $V$ be a vector space with an 
indecomposable automorphism $\phi$ on it. 
An automorphism is indecomposable if its rational canonical form has only one block, 
and the block corresponds to a power of an irreducible polynomial over $k$. 
Let $M(C,\phi)$ be the vector space $\oplus_{i=0}^{n-1}V_i$, with $V_i\iso V$, 
and let $\Lambda$ act on $M(C,\phi)$ via the schema
\[
\xymatrix@C+10pt{  
V_0 \ar@(dr,dl)[rrrrr]_{l_n=id} & V_1 \ar[l]_{l_1=\phi} &V_2 \ar[l]_{l_2=id} &\cdots\  \ar[l] & V_{n-2} \ar[l]  &V_{n-1} \ar[l]_{l_{n-1}=id}}. 
\]
Such a module is called a \dfn{module of the second kind}.
Clearly, $M(C,\phi)\iso M(C^{-1},\phi^{-1})$. 
And if $C'$ differs from $C$ by a cyclic permutation, 
say $l_1 \cdots l_n \mapsto l_n l_1 \cdots l_{n-1}$, then $M(C, \phi) \iso M(C', \phi)$. 
Moreover, if $V'$ is another vector space with an indecomposable automorphism $\phi'$,
and $V \iso V'$ via an isomorphism that commutes with $\phi$ and $\phi'$, 
then $M(C, \phi) \iso M(C', \phi')$.

\begin{thm}[{\cite[Section~4.11]{Benson}}]\label{th:class:D_n}
For any field $k$, 
the above provides a complete list of all indecomposable $\Lambda$-modules, up to isomorphism.
One of these modules has $(XY)^q-(YX)^q$ in its kernel if and only if one of the following holds:
\begin{enumerate}[(a)]
\item The module is of the first kind and the corresponding word 
does not contain $(ab)^q$, $(ba)^q$, or their inverses.

\item The module is of the second kind and no power of the corresponding word 
contains $(ab)^q$, $(ba)^q$, or their inverses.

\item The module is $M((ab)^q(ba)^{-q},id)$. 
It is a module of the second kind and is the projective indecomposable module 
for the algebra $\Lambda/((XY)^q-(YX)^q)$.

\end{enumerate}
Thus, when $k$ has characteristic $2$, a complete list of indecomposable $kD_{4q}$-modules, up to
isomorphism, consists of the $\Lambda$-modules satisfying one of these
three conditions.\qed 
\end{thm}

\begin{rmk}\label{rmk:D_n}
The fact that $kD_{4q} \iso \Lambda/((XY)^q-(YX)^q)$
yields that $kD_{4q} = M((ab)^q(ba)^{-q},id)$.
It is not hard to see from the schema of $M((ab)^q(ba)^{-q},id)$
that it has radical length $2q+1$.
Here is an illustration for $q=2$:
\begin{center}
\begin{tikzpicture}[scale=0.6]
\Dstart
\DXw
\DY
\DX
\DYw
\Dxw
\Dy
\Dx
\Dyw
\Dend 
\end{tikzpicture}.
\end{center}
The module $N = k \up_{C_2}^{D_{4q}} = kD_{4q} \otimes_{kC_2} k$ is the quotient of $kD_{4q}$
where we identify $(xy)^q$ with $1$,
in other words, $(xy)^{\frac{q}{2}}=(yx)^{\frac{q}{2}}$, for $q > 1$.
This is equivalent to $(XY)^{\frac{q}{2}}=(YX)^{\frac{q}{2}}$.
Hence $N = M((ab)^{\frac{q}{2}}(ba)^{-\frac{q}{2}},id)$
and it follows that $N$ has radical length $q+1$.
%Note that when $q=1$, we identify $x$ with $y$ and $a$ with $b$.
\end{rmk}

We want to prove that the generating number of $kD_{4q}$ does not exceed $q+1$.
Note that when $q=1$, the dihedral group $D_4$ is just $C_2 \times C_2$, and
the claim follows from Theorem~\ref{th:ab}, so we assume that $q \geq 2$ from now on unless otherwise stated.

Now let $M$ be an indecomposable $kD_{4q}$-module.
By Theorem~\ref{th:class:D_n}, it corresponds to a word
satisfying one of the conditions (a), (b) or (c).
Then $\soc(M)$ contains the submodule spanned by the vector spaces 
at positions of the form $b^{-1}a$ or $a^{-1}b$ (interpreted
cyclically if $M$ is of the second kind).
Such a position exists if $M$ is of the second kind since the condition
that the word is not a power of a smaller word forces the word to
contain both direct and inverse letters.
However, such positions are removed in $M/\soc(M)$, so the indecomposable summands of
$M/\soc(M)$ are of the first kind and correspond to words not containing $b^{-1}a$ or $a^{-1}b$.

Similarly, the indecomposable summands of $\rad(M)$ 
are of the first kind and correspond to words not containing $ba^{-1}$ or $ab^{-1}$.
It follows that the indecomposable summands of $\rad(M/\soc(M))$ 
are of the first kind and correspond to words not containing
$b^{-1}a$, $a^{-1}b$, $ba^{-1}$ or $ab^{-1}$.
Thus the words must consist entirely of direct or inverse letters.
But since $M(C) \iso M(C^{-1})$, we can assume that the words
only contain direct letters.
By (a), the possible words are $(ab)^{q-1}a$, $(ba)^{q-1}b$, or sub-words of these.
And we can prove
% The lemma fails for $q=1$.
\begin{lem}\label{le:D_4q-first-kind}
Let $M$ be a $kD_{4q}$-module of the first kind, with $q \geq 2$. If $M$ corresponds to a word
that only contains direct letters, then
its generating length is less than or equal to $q$.
\end{lem}

\begin{proof}
We are going to show that
\[\gel(M((ab)^ra)) \leq q \quad\text{and}\quad
  \gel(M((ab)^r)) \leq q \]
for $0 \leq r \leq q-1$, the case of words starting with $b$ 
being similar.

Since $D_{4q}$ is a $2$-group, the generating length of a module is always no more than
its radical length, hence its dimension.
So, for any word $C$, $\gel(M(C)) \leq \dim M(C) = |C| + 1$, where $|C|$
denotes the number of letters in $C$.
Thus we are done if $r \leq q/2 - 1$.

To handle $r \geq q/2$, we temporarily introduce the following notation for
modules with symmetry under reflection when exchanging $X$ with $Y$.
For a word $u$, write $u'$ for the inverse word with all
$a$s and $b$s exchanged, so for example $(ab^{-1}ab)' = a^{-1}b^{-1}ab^{-1}$.
Write $M'(u)$ for $M(u u')$ and $M'(u, \phi)$ for $M(u u', \phi)$.
Then $kD_{4q} = M'((ab)^q, id)$, and one can see that
$\tilde{\Omega} k = M'((b^{-1}a^{-1})^{q-1}b^{-1})$ and
$\tilde{\Omega}^{-2} k = M'((ab)^{q-1}ab^{-1})$.
It follows that we have short exact sequences
\[ 0 \to k \to \tilde{\Omega}^{-2} k \to M((ab)^{q-1}a) \oplus M((ba)^{q-1}b) \to 0 \] 
and 
\[ 0 \to k \to \tilde{\Omega} k \to M((ab)^{q-1}) \oplus M((ba)^{q-1}) \to 0. \] 
Since $q \geq 2$, one sees that $\gel(M((ab)^{q-1}a)) = \gel(M((ab)^{q-1})) \leq 2$,
which handles the case $r = q-1$.
% Their generating lengths are in fact equal to 2, but this takes a moment's thought
% and isn't needed.

Now for $r \leq q-2$, $M((ab)^r a)$ and $M((ab)^r)$ embed in $M((ab)^{q-1})$.
Thus their ghost lengths are no more than the codimension plus two, and one
can check that this is no more than $q$ when $r \geq q/2$.
\end{proof}

In general, for a $p$-group $G$ and a $kG$-module $M$,
we know that $M/ \rad(M)$ and $\soc(M)$ are sums of trivial
modules. 
Thus $\rad(M)$ is the fibre of a map $M \to \oplus k$
and $M/\soc(M)$ is the cofibre of a map $\oplus k \to M$.
So
\[\gel(M) \leq \gel(\rad(M))+1 \quad\text{and}\quad
  \gel(M) \leq \gel(M/ \soc(M))+1. \]
Hence
\[\gel(M) \leq \gel(\rad(M/ \soc(M)))+2, \]
and so by Lemma~\ref{le:D_4q-first-kind} and the discussion preceding it,
the generating number of $kD_{4q}$ does not exceed $q+2$.
This is one more than the correct answer.
We will show in Proposition~\ref{prop:length q} that the module $M((ab)^{\frac{q}{2}-1}a)$ has length $q$,
so we can't improve this bound by improving Lemma~\ref{le:D_4q-first-kind}.

We will have to be a bit more clever in the construction to get the exact generating number.
The above process takes two steps to produce a module $\rad(M/\soc(M))$ whose summands
involve only direct letters, by removing ``top'' and ``bottom'' elements.
We next show that we can add top elements instead of removing them, with the same
effect, and as a result we will be able to do both steps at the same time.

\begin{lem}\label{le:top-and-bottom}
Let $M$ be a non-projective indecomposable module, with corresponding word $C$.
There exists a short exact sequence
\[
 0 \to M \to M' \to \oplus k \to 0, 
\]
where the indecomposable summands of $M'$ are of the first kind and
correspond to words that contain no $ab^{-1}$ or $ba^{-1}$. 
\end{lem}

\begin{proof}
First suppose that $M$ is of the first kind.
If $C$ contains no $ab^{-1}$ or $ba^{-1}$, we simply set $M'$ to be $M$.
Otherwise, assume for example that $C$ contains $ab^{-1}$ and factor the word $C$ as $L_1 L_2$,
with $L_1$ ending with $a$ and $L_2$ starting with $b^{-1}$.
Write $z$ for the basis element of $M(C)$ corresponding to the vertex
connecting $L_1$ with $L_2$, and write $z_i$ for the corresponding
basis element in $M(L_i)$, $i = 1, 2$.
Then we have a short exact sequence $M \to M(L_1) \oplus M(L_2) \to k$,
where the first map takes $z$ to $z_1 - z_2$ and does 
the natural thing on the other basis elements, and
the second map takes $z_1$ and $z_2$ to $1$ in $k$ and the other basis elements to $0$.
More generally, we can write $C = L_1 L_2 \cdots L_n$, 
broken at the spots $a^{-1}b$ and $b^{-1}a$,
and set $M' = \oplus M(L_i)$.

Now suppose that $M = M(C, \phi)$ is of the second kind, where $\phi : V \to V$
is an indecomposable automorphism.
We can assume that $C = a^{-1}L\,b$ up to inverse and cyclic permutation.
Fix a basis $v_1, \ldots, v_n$ of $V$, where $n = \dim(V)$.
Let $M'' = \oplus_{i=1}^n M_i$, with each $M_i = M(C)$.
We write $w_i$ and $z_i$ for the basis elements in $M_i$ corresponding to the 
beginning and end of the word $C$.
Then we have a short exact sequence $M(C, \phi) \to M'' \to V$,
where the first map sends $v_i$ to $\phi(w_i) - z_i$ for the first vertex
and does the natural thing on the other vertices,
and the second map sends $w_i$ to $v_i$, $z_i$ to $\phi(v_i)$ and the other basis elements to $0$.
Here we regard $V$ as a module with trivial action.
Since $M''$ is a sum of modules of the first kind, we can apply the process
in the preceding paragraph to each summand to obtain a short exact sequence
$M'' \to M' \to \oplus k$, with $M'$ of the required form.
It is not hard to see that the cokernel of the composite $M \to M'' \to M'$ also has 
a trivial action,
and we are done.
\end{proof}

Note that the short exact sequence is represented by a map $\oplus \Omega k \to M$, and
this makes it possible to combine it with a map $\oplus k \to M$.

\begin{ex}
We illustrate an example for $q=2$. Write $kV$ for the module $M(a^{-1}b^{-1}ab,id_k)$:
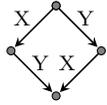
\begin{center}
\begin{tikzpicture}[scale=0.6]
\Dstart
\DXw
\DYw 
\Dxw
\Dyw
\Dend
\end{tikzpicture}.
\end{center}
We begin by defining a cofibre sequence
\[\Omega k \to kV \to M(a^{-1}b^{-1}ab) \to k.\]
To see what the maps are, first consider the module
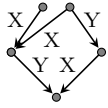
\begin{center}
\begin{tikzpicture}[scale=0.6]
\Dstart
\DX
\DYw
\Dxw
\Dy
\Draw{X}{post}{(-1.3,-1)}{below right=-2.5pt,font=\footnotesize}
\Dend
\end{tikzpicture}
\end{center}
which has $kV$ as a codimension 1 submodule.
We can choose a basis so that this becomes $M'=M(a^{-1}b^{-1}ab)$
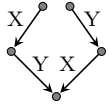
\begin{Dtikz}
\Dstart
\DX
\DYw
\Dxw
\Dy
\Dend
\end{Dtikz}
and the map $M' \to k$ takes both top points to $k$ and has kernel $kV$.
Then $M'$ corresponds to a word that does not contain $ba^{-1}$ or $ab^{-1}$, and
the summands of $M'/\soc(M') \iso M(a) \oplus M(b)$ correspond to words that only contain direct letters.
Note that the map from $k$ to $\soc(M')$ factors through $kV \to M'$, so we can combine the two steps to get a cofibre sequence
\[\Omega k \oplus k \to kV \to M(a) \oplus M(b) \to k \oplus \Sigma k.\]
By Lemma~\ref{le:D_4q-first-kind}, the generating length of the third term is at most $q$,
which is $2$ in our case.
\end{ex}

Now we are ready to prove
\begin{thm}
Let $k$ be a field of characteristic $2$. 
Then the generating number of $kD_{4q}$ is at most $q+1$, for all $q \geq 1$.
\end{thm}

\begin{proof}
The case when $q=1$ is dealt with in Theorem~\ref{th:ab}, so we prove the theorem for $q \geq 2$.

Let $M$ be a non-projective indecomposable module, with corresponding word $C$.
In the short exact sequence
$ M \to M' \to \oplus k $ from Lemma~\ref{le:top-and-bottom},
the indecomposable summands of $M'$ correspond to words
that contain no $ab^{-1}$ or $ba^{-1}$.
Hence the indecomposable summands of $M'' = M'/\soc(M')$ correspond to words of direct letters,
and $\gel(M'') \leq q$. 

We can form the octahedron
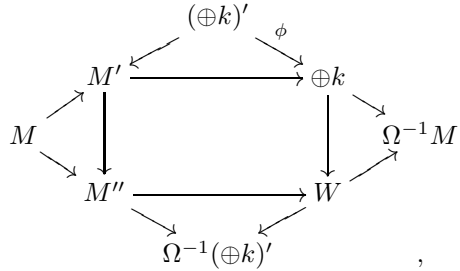
\[
\xymatrix @R=8pt @C=8pt{
                        &                     & (\oplus k)' \ar[dl] \ar[dr]^{\phi} &                          &          \\  
                        & M'  \ar[dd] \ar[rr] &                                    & \oplus k \ar[dd]\ar[dr]  &          \\   
\ \ M\  \ar[ur] \ar[dr] &                     &                                    &                          &  \Omega^{-1} M\\  
                        & M''  \ar[rr]\ar[dr] &                                    & W \ar[dl]\ar[ur]         &          \\  
                        &                     & \Omega^{-1} (\oplus k)'                 &                          &,}
\]
where $(\oplus k)'$ is $\soc(M')$.

The proof will be finished once we show that $\gel(W)=1$.
% If $\phi=0$, like the previous example or we have picked the maps carefully,
% then this is easy. In general:
Here $W$ is the cofibre of a map $\phi$ between direct sums of trivial modules.
Such a map is the sum of an identity map and a zero map.
Hence $W$ is a direct sum of trivial modules $k$ and the modules $\Omega^{-1} k$,
so $\gel(W)=1$.
\end{proof}

\begin{cor}\label{co:D_4q}
Let $k$ be a field of characteristic 2.
Then the ghost number and generating number of $kD_{4q}$ are $q+1$ for all $q \geq 1$.
\qed
\end{cor}

We now summarize and generalize the idea in the proof of the Theorem.
Suppose that we start building an object $Q$ from $P$, $Y$ and $Z$ by first using a triangle
\[ P \to X \to Y \to \Sigma P\]
and then using a triangle
\[ Q \to X \to Z \to \Sigma Q.\]
Then we can form the octahedron
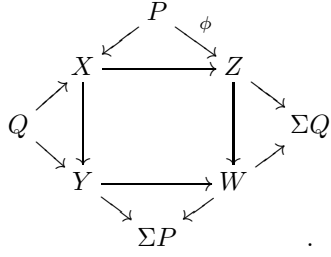
\[
\xymatrix @R=8pt @C=10pt{
                        &                    & P \ar[dl] \ar[dr]^{\phi} &                  &          \\  
                        & X  \ar[dd] \ar[rr] &                          & Z \ar[dd]\ar[dr] &          \\   
      Q \ar[ur] \ar[dr] &                    &                          &                  &  \Sigma Q\\  
                        & Y  \ar[rr]\ar[dr]  &                          & W \ar[dl]\ar[ur] &          \\  
                        &                    & \Sigma P                 &                  &.}
\]
Assume that $P$ has length $m$, $Y$ has length $n$, and $Z$ has length $l$.
Then the length of $Q$ does not exceed $m+n+l$.
Indeed, $n+\len(W)$ bounds the length of $Q$.
For example, if $\phi$ is in $\cI^s$ for some positive integer, we have
$\len(W) \leq m+l-s$ by Lemma~\ref{le:gel+ar}.
Or, if $\phi = 0$, then $W \iso Z \oplus \Sigma P$ and
the two steps can be combined.
This is analogous to the fact in topology that
when a second cell is attached to a $CW$-complex without touching a first cell,
then they can be attached to the complex at the same time.
\medskip

% Fix overfull hbox:
We finish this section by computing the generating lengths of %the modules 
$M((ab)^r)$ and $M((ab)^r a)$, with $r \leq q/2-1$.
Note that there is a category automorphism on $\StMod{kD_{4q}}$
induced by the group automorphism on $D_{4q}$ that exchanges $x$ and $y$.
It exchanges the $a$'s and $b$'s in the word which an indecomposable module 
corresponds to and preserves the ghost projective class.
As a result,
\[\gel(M((ab)^r)) = \gel(M((ba)^r)) \text{\ \ and\ \ }
  \gel(M((ab)^r a)) = \gel(M((ba)^r b))\]
for $D_{4q}$-modules with $0 \leq r \leq q-1$.

Recall from Corollary~\ref{cor:D_4q lower bound} that the module $M=kD_{2q}$ in $\StMod{kD_{4q}}$ has its generating length equal to its radical length $q+1$.
By Proposition~\ref{prop:length of rad}, $\gel(\rad (M/ \soc(M)))=\gel(M)-2=q-1$.
Note that $M = M((ab)^{l+1}(a^{-1}b^{-1})^{l+1},id)$, where $l = q/2 - 1$, so
$\rad (M/ \soc(M)) \iso M((ab)^l) \oplus M((ba)^l)$.
Then, since exchanging $a$'s and $b$'s preserves the generating length,
\[\gel (M((ab)^l)) = \gel(M((ba)^l)) = q-1. \]
It follows that
\[\gel(M((ab)^r)=2r+1 \text{ if } r \leq l, \]
and
\[\gel(M((ab)^ra)=2(r+1) \text{ if } r \leq l-1.\]

We need to be a bit trickier to handle the module $M((ab)^l a)$.
\begin{pro}\label{prop:length q}
The $kD_{4q}$-module $M((ab)^la)$ has generating length $q$, where $l = q/2 - 1$. 
\end{pro}

\begin{proof}
We have a triangle
\[ \Sigma k \oplus k \to M \to M((ab)^la) \oplus M((ba)^lb),\] 
where the map $\Sigma k \to M$ is a surjection.

Hence $\gel(M((ab)^la) \oplus M((ba)^lb)) \geq q$.  Since its radical length
is $q$, this must be an equality.
Then, using the symmetry again,
\[\gel(M((ab)^la)) = \gel(M((ba)^lb))=q. \qedhere \]
\end{proof}

\subsection{Ghost number of $C_{p^r} \times C_{p^s}$}\label{ss:gn-C3xC3}

Let $G = C_{p^r} \times C_{p^s}$.
In this section we show that
\[
\text{the ghost number of }kG \leq
\text{the generating number of }kG \leq
p^r+p^s-3
\]
and give the exact result when $p^r$ is $3$ or $4$.
Note that a general upper bound for the generating number for a $p$-group
is given by the radical length of $kG$ minus $1$ (Theorem~\ref{th:finite gl}).
This gives $p^r+p^s-2$ for the group $C_{p^r} \times C_{p^s}$, and
our result refines this upper bound by $1$.
To keep the indices simple, we give a detailed proof 
for the group $C_3 \times C_3$ at the prime $3$, and
we indicate how to modify the proof to cover the general case.
We are going to show that the composite of any three ghosts is stably trivial for the group $C_3 \times C_3$,
using Theorem~\ref{th:class:D_n}.

Here is an overview of our strategy.
Given a finitely generated projective-free module $N$ with radical length $n$
and an $l$-fold ghost $g: N \to N_1$ in $\Mod{kG}$,
where $N_1$ is an arbitrary projective-free module,
we can form the following commutative diagram:
\[
\xymatrix{ 
N \ar[d]^{p_1} \ar[r]^{g}    & N_1 \\
N/\rad^{n-l}(N) \ar[r]^-{p_2} & N/\soc^l(N) \ar[u]^h.}
\]
The $l$-fold ghost $g$ factors through $N / \soc^l(N)$ by Corollary~\ref{co:soc}, and
the canonical projection $N \to N / \soc^l(N)$ factors through $N / \rad^{n-l}(N)$
because $\rad^{n-l}(N) \subseteq \soc^l(N)$.
If we have good control over the modules $N / \rad^{n-l}(N)$ or
$N / \soc^l(N)$,
we can factorize a long composite of ghosts as an $l$-fold ghost $g: N \to N_1$
followed by another composite of ghosts $f : N_1 \to N_2$, and
check whether $f$ is stably trivial on $N / \rad^{n-l}(N)$ or
$N / \soc^l(N)$.
For example, we can take $l$ to be $n-1$, so that $N / \rad(N)$
is a sum of trivial modules.
Hence, if the map $f$ is a ghost, the composite $f \circ g$ is stably trivial,
and so we have reproved that the generating length of $N$ is at most its radical length $n$
(Theorem~\ref{th:finite gl}).
If we want to improve the bound, we need to choose $l$ smaller.
We will take $l = n-2$.

While we can't classify $k(C_{p^r} \times C_{p^s})$-modules, we can use
Theorem~\ref{th:class:D_n} to classify certain quotient modules.
We use that there is an isomorphism
$k(C_{p^r} \times C_{p^s}) \iso k[X,Y]/(X^{p^r}, Y^{p^s})$, where $X = x-1$ and $Y = y-1$,
and $x$ and $y$ are the generators of the cyclic summands.
Under this isomorphism, $\rad(k(C_{p^r} \times C_{p^s})) \iso (X,Y)$
and $\rad^2(k(C_{p^r} \times C_{p^s})) \iso (X^2, XY, Y^2)$.
Therefore $k(C_{p^r} \times C_{p^s})/\rad^2(k(C_{p^r} \times C_{p^s})) \iso \Lambda'$,
where $\Lambda' = \Lambda/(XY, YX) \iso k[X,Y]/(X^2,Y^2,XY)$
and $\Lambda = k\langle X,Y \rangle / (X^2,Y^2)$ is the ring from Section~\ref{ss:gn-D4q}.
Thus when $M$ is a $k(C_{p^r} \times C_{p^s})$-module, $M/\rad^2(M)$ will be a $\Lambda'$-module.
Up to isomorphism, the indecomposable $\Lambda'$-modules biject with the
$\Lambda$-modules of Theorem~\ref{th:class:D_n} satisfying
conditions (a) or (b) for $q=1$.
Condition (c) is excluded by the requirement that $XY$ be in the kernel.

Our proof will use this classification, so we will make it more explicit.
A module satisfying condition (a) is of the first kind.
If it has odd dimension, it 
is either the trivial module $k$;
the module $M((b^{-1}a)^n)$ for some positive integer $n$, 
which we say has shape ``$W$'';
or the module $M((ab^{-1})^n)$ for some positive integer $n$, 
which we say has shape ``$M$''.
For example, the ``$M$'' module $M((ab^{-1})^3)$ looks like
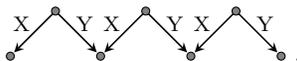
\begin{center}
\begin{tikzpicture}[scale=0.6]
\Dstart
\Dxw \DYw
\Dxw \DYw
\Dxw \DYw
\Dend
\end{tikzpicture}\ .
\end{center}
A module of the first kind with even dimension
is one of the above with one end removed.

One can check that a module satisfying condition (b) of Theorem~\ref{th:class:D_n}
corresponds to the word $b^{-1}a$,
up to inverse and cyclic permutation.
% It isn't allowed to be a power of another word.
Recall that the additional data one needs to specify are a
vector space $V$ with an indecomposable automorphism $\phi$.
Since $\phi$ is indecomposable, one can choose a basis 
$\{ v_1, v_2, \ldots, v_m \}$ for $V$ such that $\phi(v_i) = v_{i+1}$ for $i < m$.
Thus we can view such a module as a quotient of an ``$M$'' module,
with a relation that identifies the right bottom basis element
with a linear combination of the other bottom basis elements,
as specified by $\phi(v_m)$.
% This is possible because by definition of irreducible, $\phi$
% has rational canonical form with one block.

We point out that this is very similar to the
classification of $kV$-modules given in~\cite[Theorem 4.3.3]{Benson},
where $k$ has characteristic $2$.

Recall that the radical length of $k(C_{p^r} \times C_{p^s})$ is $p^r+p^s-1$.
If $N$ is projective-free, then its radical length $n$ is at most $p^r+p^s-2$, so
we pick $l = p^r+p^s-4$.
Note that $N / \rad^2(N)$ and $N / \soc^l(N)$ are naturally $\Lambda'$-modules.
And we have the following lemma, which helps describe summands of $N / \soc^l(N)$.

\begin{lem}\label{le:rank 2}
Let $G = C_{p^r} \times C_{p^s}$ be an abelian $p$-group of rank $2$ with generators $x$ and $y$, respectively, and
let $k$ be a field of characteristic $p$.
Write $X=x-1$ and $Y=y-1$ in $kG$,
and let $l = p^r+p^s-4$.
Suppose $M$ is a $kG$-module containing elements $z_0$, $z_2$, and $z_4$
such that $Y z_0 - X z_2$ and $Y z_2 - X z_4$ are in $\soc^l(M)$.
If $p^s \geq 3$, then $X z_0$ and $X z_2$ are in $\soc^l(M)$.
Similarly, if $p^r \geq 3$, then $Y z_2$ and $Y z_4$ are in $\soc^l(M)$.
\end{lem}

% This is clearly not true when $p^r=p^s=2$.
Intuitively, this is saying that we cannot have a ``W''-shape in the module
$M / \soc^l(M)$.
In particular, only $k$, $M(ab^{-1})$ and $M((ab^{-1})^2)$ can appear as indecomposable summands of $M/\soc^l(M)$
if $M$ is projective-free and $p^r, p^s \geq 3$.
Note that to exclude a module like $M(a)$, one takes $z_2 = z_4 = 0$,
so the ``W'' isn't visible in this case.

\begin{proof}
Assume that $p^s \geq 3$.
To show that $X z_0 \in \soc^l(M)$, we need to show that it is
killed by $\rad^l(kG)$, which is generated by $X^{p^r-1}Y^{p^s-3}$,
$X^{p^r-2}Y^{p^r-2}$ and $X^{p^r-3}Y^{p^s-1}$ 
(where the last one is omitted if $p^r = 2$).
We compute
\begin{gather*}
X^{p^r-1} Y^{p^s-3} X z_0 = X^{p^r} Y^{p^s-3} z_0 = 0 ,\\
X^{p^r-2} Y^{p^r-2} X z_0 = X^{p^r-1} Y^{p^s-3} Y z_0 = X^{p^r} Y^{p^s-3} z_2 = 0 , \\
\intertext{and}
X^{p^r-3} Y^{p^s-1} X z_0 = X^{p^r-2} Y^{p^s-2} Y z_0
= X^{p^r-1} Y^{p^s-3} Y z_2 = X^{p^r} Y^{p^s-3} z_4 = 0 ,
\end{gather*}
where we have made used of fact that $Y z_0 - X z_2$ and $Y z_2 - X z_4$ are
killed by the generators.
Hence $X z_0 \in \soc^l(M)$.
Similarly,
\begin{gather*}
X^{p^r-1} Y^{p^s-3} X z_2 = 0, \\
X^{p^r-2} Y^{p^s-2} X z_2 = X^{p^r-1} Y^{p^s-3} Y z_2 = X^{p^r} Y^{p^r-3} z_4 = 0, \\
\intertext{and}
X^{p^r-3} Y^{p^s-1} X z_2 = X^{p^r-3} Y^{p^s-1} Y z_0 = 0.
\end{gather*}
Hence $X z_2 \in \soc^l(M)$.
The other case is symmetrical.
\end{proof}

We are now ready to prove the main theorem.

\begin{thm}\label{th:C_3xC_3}
Let $G = C_3 \times C_3$, and
let $k$ be a field of characteristic $3$.
Then the ghost number of $kG$ is $3$.
\end{thm}

\begin{proof}
As above, we write $x$ and $y$ for generators of the two factors
of $C_3 \times C_3$, and let $X = x - 1$ and $Y = y - 1$.

Theorem~\ref{th:ab} gives a lower bound of $3$, so it suffices to
show that the composite of any three ghosts in $\Mod{kG}$ 
out of a finitely-generated module is stably trivial.
As we have explained, we consider the diagram
\[
\xymatrix{ 
N \ar[d]^{p_1} \ar[r]^{g_1}  & N_1 \ar[r]^{g_2}  & N_2\ar[r]^{g_3}  & N_3\\
N/\rad^2(N) \ar[rr]^{p_2} && N/\soc^2(N) \ar[u]^h,}
\]
where $g_1$, $g_2$, and $g_3$ are ghosts in $\Mod{kG}$ and
$N$, $N_1$, $N_2$, and $N_3$ are projective-free.
Note that this diagram commutes in the module category.
We will show that the composite $g_3 \circ h \circ p_2$ is stably trivial,
by restricting to each indecomposable summand $M$ of $N / \rad^2(N)$.
We divide the summands $M$ into four cases, 
and write $j$ for the inclusion map $M \to N / \rad^2(N)$.
\medskip

\noindent
\textbf{Case 1:} $M$ is not of the form $k$, $M(ab^{-1})$ or $M((ab^{-1})^2)$.

We claim that $\soc(M) \subseteq \ker (p_2 \circ j)$, hence 
$p_2 \circ j$ factors through a sum of trivial modules.
Therefore, since $g_3$ is a ghost,
the composite $g_3 \circ  h \circ p_2 \circ j$ is stably trivial.
We actually show that $p_1^{-1} j(\soc (M)) \subseteq \soc^2(N)$,
which suffices, since $p_2$ kills $\soc^2(N)$.
Observe using the classification that since $M$ is not $k$, $M(ab^{-1})$ or $M((ab^{-1})^2)$,
the elements $X(z_0)$, $X(z_2)$,  $Y(z_2)$ and $Y(z_4)$ span $\soc(M)$ as
$z_0$, $z_2$, and $z_4$ vary over elements satisfying
$Y(z_0) = X(z_2)$ and $Y(z_2) = X(z_4)$.
Suppose that we have $s \in p_1^{-1}j (\soc (M))$, say
$p_1 (s) = j (X(z_0))$ for some $z_0 \in M$ satisfying the above relations.
Since $p_1$ is surjective, we have $\tilde{z_0}$, $\tilde{z_2}$, and $\tilde{z_4} \in N$ that
project to $j(z_0)$, $j(z_2)$, and $j(z_4)$, respectively.
Then $p_1 (Y (\tilde{z_0})) = p_1 (X (\tilde{z_2}))$ and
$p_1 (Y (\tilde{z_2})) = p_1 (Y (\tilde{z_4}))$.
Since $N$ is projective-free, its radical length is at most $4$,
hence $\rad^2 (N) \subseteq \soc^{2} (N)$.
Now we can apply Lemma~\ref{le:rank 2} and see that $X (\tilde{z_0}) \in \soc^{2} (N)$.
It follows that $s \in \soc^{2} (N)$ because $p_1 (s) = p_1 (X (\tilde{z_0}))$.
The other cases when 
$p_1 (s) = j (Xz_2)$, $j (Yz_2)$, or $j (Yz_4)$ are similar.
\medskip

\noindent
\textbf{Case 2:} $M = M(ab^{-1})$.

The map $p_1$ is surjective, so $g_3 h p_2$ has its image in $\rad^3(N_3)$,
using Corollary~\ref{co:soc} and the fact that the diagram commutes in $\Mod{kG}$.
$M$ has a basis $\{ z, Xz, Yz \}$ for some $z$
and the map $g_3 h p_2$ sends $z$ to an element of the form $X^2 Y w_1 + X Y^2 w_2$.
After restriction to $M$, $g_3 h p_2$ factors through the injective module 
which is free on two generators $v_1$ and $v_2$ via the maps
sending $z$ to $X^2 Y v_1 + X Y^2 v_2$, $v_1$ to $w_1$ and $v_2$ to $w_2$.
Thus $g_3 h p_2$ is stably trivial on $M$.
\medskip

\noindent
\textbf{Case 3:} $M = M((ab^{-1})^2)$.

The module $M((ab^{-1})^2)$ has schema
$ kz_0 \xleftarrow{X} kz_1 \xrightarrow{Y} kz_2 \xleftarrow{X} kz_3 \xrightarrow{Y} kz_4 $.
By considering the injective hull of $M((ab^{-1})^2)$, which is free on three generators, we see that
a map out of it is stably trivial if it sends
$z_1$ to $XY^2 w_1 + X^2Y w_2$ and
$z_3$ to $XY^2 w_2 + X^2Y w_3$ for some elements $w_1$, $w_2$, and $w_3$.
This is equivalent to 
$z_1$ being sent to $X \alpha$ and $z_3$ being sent to $Y \alpha$ for some $\alpha$ in the $2^{nd}$ radical.

To prove that this is the case, we form the following diagram:
\[
\xymatrix@C=29pt{ 
\tilde{\Omega}^{-2}k \ar@{-->}[r]^{f} &N \ar[d]^{p_1} \ar[r]^{g_1}  & N_1 \ar[r]^{g_2}  & N_2\ar[r]^{g_3}  & N_3\\
M((ab^{-1})^2) \ar[r]^j & N/\rad^2(N) \ar[rru]^{h p_2}}.
\]
Writing $g = g_3 \circ g_2 \circ g_1$,
we will show below that we can choose $\tilde{z}_1$ and $\tilde{z}_3$ in $N$ with
\[
g(\tilde{z}_1)=g_3 h p_2 j(z_1),\quad
g(\tilde{z}_3)=g_3 h p_2 j(z_3),
\quad\text{and}\quad Y \tilde{z}_1 = X \tilde{z}_3.
\]
Since $\tilde{\Omega}^{-2}k$ is the free module on two generators $u_1$ and $u_2$
subject to the relation $Y u_1 = X u_2$,
the last displayed equality allows us to construct the dotted map $f$,
by sending the generators to $\tilde{z}_1$ and $\tilde{z}_3$, respectively.
We will now show that
\[
g(\tilde{z}_1) = X \alpha \quad\text{and}\quad g(\tilde{z}_3) = Y \alpha
\]
for some $\alpha \in \rad^2 (N_3)$. 
Since $g_1$ is a ghost, the composite $g_1 f$ is stably trivial.
It follows that, modulo $\soc^2 (N_1)$, 
% I think it's \soc^2 above because any map kG --> N_1 sends
% soc^r kG into soc^{r-1} N_1, since N_1 is projective free.
$g_1(\tilde{z}_1) = X \alpha' \text{ and } g_1(\tilde{z}_3) = Y \alpha'$
for some $\alpha' \in N_1$.
Since $g_3 g_2$ is a double ghost, it kills $\soc^2 (N_1)$ and
takes $\alpha'$ into $\rad^2 (N_3)$.
Hence we can set $\alpha = g_3 g_2 (\alpha')$.

We still need to pick the $\tilde{z}_1$ and $\tilde{z}_3$.
First choose $\tilde{z}_1'$ and $\tilde{z}_3'$ in $N$ that project to $j(z_1)$ and $j(z_3)$ in $M((ab^{-1})^2)$, respectively. 
The difference $Y\tilde{z}_1'-X\tilde{z}_3'$ is in $\rad^2(N)$, 
say $Y\tilde{z}_1'-X\tilde{z}_3'=Y\beta - X\gamma$ for some $\beta$ and $\gamma \in \rad(N)$. 
We set $\tilde{z}_1=\tilde{z}_1'-\beta$ and $\tilde{z}_3=\tilde{z}_3'-\gamma$ 
so that $Y\tilde{z}_1=X\tilde{z}_3$. 
By Corollary~\ref{co:soc}, $g(\beta)=g(\gamma)=0$, hence
\[
g(\tilde{z}_1)=g(\tilde{z}_1')=g_3 h p_2 j(z_1) \quad\text{and}\quad
g(\tilde{z}_3)=g(\tilde{z}_3')=g_3 h p_2 j(z_3).
\]

\noindent
\textbf{Case 4:} $M = k$, the trivial module.

Then clearly $g_3 \circ h \circ p_2$ is stably trivial when restricted to $M$,
since $g_3$ is a ghost.
\end{proof}
% It is not clear how to build a module in three steps for the group C_3 x C_3.

Since we don't require the modules $N_1$, $N_2$, and $N_3$ to be finitely-generated in the proof,
we have actually proved a stronger result, a bound for the generating number, giving:
\begin{cor}\label{co:C_3xC_3}
Let $k$ be a field of characteristic $3$. 
Then the generating number of $k(C_3\times C_3)$ is~$3$.\qed
\end{cor}

\begin{rmk}\label{rmk:rank 2}
The arguments in this section go through for the group $G = C_{p^r} \times C_{p^s}$
with $2 < p^r \leq p^s$, and we get that
the generating number of $kG$ is less than or equal to $p^r+p^s-3$.
Theorem~\ref{th:ab} gives a lower bound of 
$\left\lceil \frac{p^r-1}{2} \right\rceil + p^s - 1$.
In particular, if $p^r=3$, the ghost number of $kG$ is $p^s$, and
if $p^r=4$, the ghost number of $kG$ is $p^s+1$.

We now indicate the modifications needed in the proof of the general case.
Instead of $g_2$ being a ghost, we take it to be a $(p^r+p^s-5)$-fold ghost.
Then the map $h$ has domain $N / \soc^{p^r+p^s-4} (N)$.
In Case 1, one checks that $p_1^{-1}j(\soc (M)) \subseteq \soc^{p^r+p^s-4} (N)$.
In Case 2, the map $g_3 h p_2$ sends $z \in M(ab^{-1})$ to an element of the form $X^{p^r-1} Y^{p^s-2} w_1 + X^{p^r-2} Y^{p^s-1} w_2$.
In Case 3, a map out of $M((ab^{-1})^2)$ is stably trivial if it sends
$z_1$ to $X \alpha$ and
$z_3$ to $Y \alpha$ for some $\alpha$ in the $(p^r+p^s-4)^{th}$ radical.
Case 4 is unchanged.
\end{rmk}

\subsection{Possible ghost numbers for group algebras}\label{ss:possible}

In this section, we classify group algebras with certain small ghost numbers,
and also put constraints on which ghost numbers can occur.
Whenever we write $kG$, $k$ can be any field whose characteristic divides
the order of $G$.

In~\cite{Gh in rep} it is shown that the abelian groups $G$ such that
the ghost number of $kG$ is $2$ are $C_4$, $C_2 \times C_2$ and $C_5$.
The results of the previous section and Theorem~\ref{th:ab} give
a complete list of abelian $p$-groups of ghost number $3$:

\begin{pro}\label{pr:abelian-gn-3}
Let $G$ be an abelian $p$-group. Then the ghost number of $kG$ is $3$ if and only if
$G$ is $C_7$, $C_3 \times C_3$, or $C_2 \times C_2 \times C_2$ if and only if
the generating number of $kG$ is $3$. \qed
\end{pro}

Below we will extend this to non-abelian $p$-groups, with one ambiguous group.
We first recall a consequence of Jennings' formula which will also be useful
in studying the gaps
in the possible ghost numbers.

\begin{lem}[{\cite[Thm.~3.14.6]{Benson}}]\label{le:Jen}
Let $k$ be a field of characteristic $p$. 
If $G$ is a group of order $p^r\!$, then
\[ \text{nilpotency index of } J(k(C_{\!p}^r)) \leq
   \text{nilpotency index of } J(kG)  \leq
   \text{nilpotency index of } J(k(C_{p^r})) . \xqedhere{17.7pt}\]
\end{lem}

Note that the nilpotency index of $J(k(C_{\!p}^{r}))$ is $r(p-1)+1$.

\begin{pro}\label{pr:lower-bound}
Let $k$ be a field of characteristic $p$.
If $G$ is a group of order $p^r\!$, then
the ghost number of $kG$ is at least $(r-1)(p-1)+1$.
\end{pro}

\begin{proof}
The group $G$ has a quotient $H$ of order $p^{r-1}$.
By Theorem~\ref{th:rl=gel}, $\rl(kH)$ is a lower bound for the ghost number of $kG$.
Now by the previous lemma, $\rl(kH) \geq (r-1)(p-1)+1$,
so we are done.
\end{proof}

\begin{thm}
The following is a complete list of the $p$-groups $G$ such that $kG$ has 
the specified ghost number:
\begin{itemize}
\item [1:] the abelian groups $C_2$ and $C_3$;
\item [2:] the abelian groups $C_4$, $C_2 \times C_2$ and $C_5$;
\item [3:] the abelian groups $C_7$, $C_3 \times C_3$ and $C_2 \times C_2 \times C_2$,
           the dihedral group $D_8$ of order 8,
           and possibly the quaternion group $Q_8$, which has ghost number 3 or 4.
\end{itemize}
In each case, except possibly for $Q_8$, the generating number equals the ghost number.
\end{thm}

\begin{proof}
The case of ghost number 1 is the main result of~\cite{GH for p}.

A non-abelian $p$-group must have order $p^r$ for $r \geq 3$, so by
Proposition~\ref{pr:lower-bound} it must have ghost number at least 3.
Thus a $p$-group of ghost number 2 must be abelian, and this case is
proved in~\cite{Gh in rep}.

The only ways for $(r-1)(p-1)+1$ to equal 3 are $p^r = 8$ or 9.
The non-abelian groups of order 8 are $D_8$ and $Q_8$,
which are discussed in Corollary~\ref{co:D_4q}, Theorem~\ref{th:C_3xC_3}
and Corollary~\ref{co:C_3xC_3},
and there are no non-abelian groups of order 9.
The abelian case is Proposition~\ref{pr:abelian-gn-3}.
\end{proof}

Next we observe that, for a fixed prime $p$,
not all positive integers can be the ghost number of some $p$-group.
For example, since the generating hypothesis fails for $p>3$,
the number $1$ cannot be the ghost number of a $p$-group with $p>3$.
On the other hand, the elementary abelian $2$-group of rank $l$
has ghost number $l-1$,
so every positive integer can be a ghost number at the prime $2$.
Here is a result giving gaps in the possible ghost numbers at odd primes.

\begin{thm}
Let $p$ be an odd prime, and let $k$ be a field of characteristic $p$.
Write $(l_1, l_2, l_3, \ldots)$ for the increasing sequence
of integers that are ghost numbers of the group algebras $kG$,
with $G$ being a $p$-group.
Then $l_1 = \frac{p-1}{2}$,
\[\frac{3(p-1)}{2} \leq l_2 = \text{ghost number of }C_p \times C_p \leq 2p-3,\] 
and $\min(\frac{p^2-1}{2},2p-1) \leq l_3$.
\end{thm}

\begin{proof}
We know that
the ghost number of $C_p$ is $\frac{p-1}{2}$ and
that of $C_{p^2}$ is $\frac{p^2-1}{2}$~\cite[Thm.~5.4]{Gh in rep}.
And the ghost number of $C_p \times C_p$ is constrained by
Theorem~\ref{th:ab} and Remark~\ref{rmk:rank 2}:
\[\frac{3(p-1)}{2} \leq \text{ghost number of }C_p \times C_p \leq 2p-3 . \] 
By Proposition~\ref{pr:lower-bound}, the groups of order $p^r$ with $r \geq 3$ have
ghost numbers at least $2p-1$.
Comparing these numbers, we get
\[ \frac{p-1}{2} \leq 2p-3 \leq \min(\frac{p^2-1}{2},2p-1),\]
and the theorem follows.
\end{proof}

Thus one sees that for large primes there are large gaps in the
sequence of possible ghost numbers.

Observe that when $p \geq 5$,
\[
\text{the ghost number of } k(C_{\!p}^3) \leq 3p-3 \leq \frac{p^2-1}{2} =
\text{the ghost number of } kC_{p^2},
\] 
where the first inequality uses Theorem~\ref{th:finite gl}.
And by Theorem~\ref{th:ab} and Proposition~\ref{pr:lower-bound},
the ghost number of $k(C_{\!p}^r)$ is no more than the ghost number of any $p$-group with larger size.
% In fact, it is strictly less than above (but not below).
We conjecture that this is also true for groups of the same size,
which would imply that $l_3$ is the ghost number of $k(C_{\!p}^3)$ when $p \geq 5$.
The following conjecture should be viewed as the stabilized version of Lemma~\ref{le:Jen}.

\begin{conj}
Let $k$ be a field of characteristic $p$. 
If $G$ is a $p$-group of order $p^r$, then
\[ \text{ghost number of } k(C_{\!p}^r) \leq
   \text{ghost number of } kG       \leq
   \text{ghost number of } k(C_{p^r}) . \]
\end{conj}


\begin{thebibliography}{bib}

\bibitem{A-R}
M. Auslander and I. Reiten.
Representation theory of Artin algebras IV.
{\em Comm. in Algebra}, \textbf{5}(5) (1977), 443--518.

\bibitem{Benson}
D. J. Benson.
{\em Representations and cohomology I}.
Cambridge Univ. Press, Cambridge, 1998.

\bibitem{GH for p}
D. J. Benson, S. K. Chebolu, J. D. Christensen and J. Min\'{a}\v{c}.
The generating hypothesis for the stable module category of a $p$-group.
{\em Journal of Algebra} \textbf{310}(1) (2007), 428--433.

\bibitem{BoVdB}
A. Bondal and M. Van den Bergh.
Generators and representability of functors in commutative and
noncommutative geometry.
{\em Moscow Math. J.} \textbf{3} (2003), 1--36.

\bibitem{GH general}
A. M. Bohmann and J. P. May.
A presheaf interpretation of the generalized Freyd conjecture.
% Used to be called: Thinking about the Freyd conjecture.
{\em Theory Appl. Categ.} \textbf{26}(16) (2012), 403--411. 

\bibitem{Carlson}
J. F. Carlson.
{\em Modules and group algebras}. 
Lectures in Mathematics, ETH Z\"{u}rich. Birkh\"{a}user Verlag, Basel, 1996.
Notes by Ruedi Suter.

\bibitem{GH split}
J. F. Carlson, S. K. Chebolu and J. Min\'{a}\v{c}.
Freyd's generating hypothesis with almost split sequences.
{\em Proc. Amer. Math. Soc.} \textbf{137} (2009), 2575--2580.

\bibitem{admit}
S. K. Chebolu, J. D. Christensen and J. Min\'{a}\v{c}.
Groups which do not admit ghosts.
{\em Proc. Amer. Math. Soc.} \textbf{136}(4) (2008), 1171--1179.

\bibitem{Gh in rep}
S. K. Chebolu, J. D. Christensen and J. Min\'{a}\v{c}.
Ghosts in modular representation theory.
{\em Advances in Mathematics} \textbf{217}(6) (2008), 2782--2799.

\bibitem{GH per}
S. K. Chebolu, J. D. Christensen and J. Min\'{a}\v{c}.
Freyd's generating hypothesis for groups with periodic cohomology.
{\em Canadian Mathematical Bulletin}, \textbf{55}(1) (2012), 48--59.

\bibitem{Chr}
J. D. Christensen.
Ideals in triangulated categories: phantoms, ghosts and skeleta.
{\em Advances in Mathematics} \textbf{136}(2) (1998), 284--339.

\bibitem{CW}
J. D. Christensen and G. Wang.
Ghost numbers of group algebras II.
Preprint, \texttt{arXiv:1310.5682}, 2013.

\bibitem{freydGH}
P. Freyd.
Stable homotopy.
In {\em Proc. Conf. Categorical Algebra (La Jolla, Calif., 1965)},
pages 121--172. Springer, New York, 1966.

\bibitem{Jennings}
S. A. Jennings.
The structure of the group ring of a $p$-group over a modular field.
{\em Trans. Amer. Math. Soc.} \textbf{50} (1941), 175--185.

\bibitem{Krause}
H. Krause.
Auslander-Reiten theory via Brown representability.
{\em K-Theory} \textbf{20}(4) (2000), 331--344.

\end{thebibliography}
\end{document}